\documentclass{amsart}
\usepackage{amssymb}
\usepackage{amsfonts}
\usepackage{hyperref}

\newtheorem{theorem}{Theorem}
\theoremstyle{plain}

\newtheorem{corollary}{Corollary}

\newtheorem{lemma}{Lemma}

\newtheorem{proposition}{Proposition}
\newtheorem{remark}{Remark}

\numberwithin{equation}{section}
\input{tcilatex}

\begin{document}
\title[Monotonicity rules for the ratio of two Laplace transforms]{The
monotonicity rules for the ratio of two Laplace transforms with applications}
\author{Zhen-Hang Yang and Jing-Feng Tian}
\address{Zhen-Hang Yang, College of Science and Technology\\
North China Electric Power University, Baoding, Hebei Province, 071051, P.
R. China and Department of Science and Technology, State Grid Zhejiang
Electric Power Company Research Institute, Hangzhou, Zhejiang, 310014, China}
\email{yzhkm\symbol{64}163.com}
\address{Jing-Feng Tian\\
College of Science and Technology\\
North China Electric Power University\\
Baoding, Hebei Province, 071051, P. R. China}
\email{tianjf\symbol{64}ncepu.edu.cn}
\subjclass[2010]{Primary 44A10, 26A48; Secondary 33B15, 33C10}
\keywords{Laplace transform, monotonicity rule, psi function, modified
Bessel functions of the second kind}
\thanks{This paper is in final form and no version of it will be submitted
for publication elsewhere.}

\begin{abstract}
\end{abstract}

\begin{abstract}
Let $f$ and $g$ be both continuous functions on $\left( 0,\infty \right) $
with $g\left( t\right) >0$ for $t\in \left( 0,\infty \right) $ and let $%
F\left( x\right) =\mathcal{L}\left( f\right) $, $G\left( x\right) =\mathcal{L%
}\left( g\right) $ be respectively the Laplace transforms of $f$ and $g$
converging for $x>0$. We prove that if there is a $t^{\ast }\in \left(
0,\infty \right) $ such that $f/g$ is strictly increasing on $\left(
0,t^{\ast }\right) $ and strictly decreasing on $\left( t^{\ast },\infty
\right) $, then the ratio $F/G$ is decreasing on $\left( 0,\infty \right) $
if and only if%
\begin{equation*}
H_{F,G}\left( 0^{+}\right) =\lim_{x\rightarrow 0^{+}}\left( \frac{F^{\prime
}\left( x\right) }{G^{\prime }\left( x\right) }G\left( x\right) -F\left(
x\right) \right) \geq 0,
\end{equation*}%
with%
\begin{equation*}
\lim_{x\rightarrow 0^{+}}\frac{F\left( x\right) }{G\left( x\right) }%
=\lim_{t\rightarrow \infty }\frac{f\left( t\right) }{g\left( t\right) }\text{
\ and \ }\lim_{x\rightarrow \infty }\frac{F\left( x\right) }{G\left(
x\right) }=\lim_{t\rightarrow 0^{+}}\frac{f\left( t\right) }{g\left(
t\right) }
\end{equation*}%
provide the indicated limits exist. While $H_{F,G}\left( 0^{+}\right) <0$,
there is at leas one $x^{\ast }>0$ such that $F/G$ is increasing on $\left(
0,x^{\ast }\right) $ and decreasing on $\left( x^{\ast },\infty \right) $.
As applications of this monotonicity rule, a unified treatment for certain
bounds of psi function is presented, and some properties of the modified
Bessel functions of the second are established. These show that the
monotonicity rules in this paper may contribute to study for certain special
functions because many special functions can be expressed as corresponding
Laplace transforms.
\end{abstract}

\maketitle

\section{Introduction}

The Laplace transform of a function $f(t)$ defined on $[0,\infty )$ is the
function $F(s)$, which is a unilateral transform defined by%
\begin{equation*}
F(s)=\int_{0}^{\infty }f\left( t\right) e^{-st}dt,
\end{equation*}%
where $s$ is a complex number frequency parameter. The Laplace transform of
a function $f(t)$ is also denoted by $\mathcal{L}\left( f\right) $.

It is known that some special functions can be represented as corresponding
Laplace transforms, for example, Binet formula for gamma function:%
\begin{equation*}
\ln \Gamma \left( z\right) -\left( z-\frac{1}{2}\right) \ln z+z-\frac{1}{2}%
\ln \left( 2\pi \right) =\int_{0}^{\infty }\left( \frac{1}{e^{t}-1}-\frac{1}{%
t}+\frac{1}{2}\right) \frac{e^{-zt}}{t}dt\text{ \ }\mathbb{R}\left( z\right)
>0
\end{equation*}%
(see \cite[p. 21, Eq. (5)]{Erdelyi-HTF-I-1981}); the integral representation
of the modified Bessel functions of the second (see \cite[p. 181]%
{Watson-ATTBF-CUP-1922})%
\begin{equation}
K_{v}\left( x\right) =\int_{0}^{\infty }e^{-x\cosh t}\cosh \left( vt\right)
dt,  \label{Kv-Ir}
\end{equation}%
which, by replacing $\cosh t-1$ with $t$, can be expressed as%
\begin{equation*}
K_{v}\left( x\right) =e^{-x}\int_{0}^{\infty }e^{-xt}\frac{\cosh \left( v%
\func{arccosh}\left( t+1\right) \right) }{\sqrt{t\left( t+2\right) }}dt;
\end{equation*}%
the Gaussian Q-function \cite{Simon-PDGRV-S-2006} defined by%
\begin{equation*}
Q\left( x\right) =\frac{1}{\sqrt{2\pi }}\int_{x}^{\infty }e^{-t^{2}/2}dt%
\text{ for }x>0\text{,}
\end{equation*}%
which, by a change of variable $t=u+x$, is represented as%
\begin{equation*}
Q\left( x\right) =\frac{1}{\sqrt{2\pi }}e^{-x^{2}/2}\int_{0}^{\infty
}e^{-u^{2}/2}e^{-ux}du.
\end{equation*}%
More examples can be found in \cite{Miller-ITSF-12-2001}, \cite%
{Magnus-FTSFMPS-SV-1966}.

An important notion related to the Laplace transform is the completely
monotonic functions. A function $F$ is said to be completely monotonic on an
interval $I$, if $F$ has derivatives of all orders on $I$ and satisfies%
\begin{equation}
(-1)^{n}F^{(n)}(x)\geq 0\text{ for all }x\in I\text{ and }n=0,1,2,....
\label{cm}
\end{equation}%
If the inequality (\ref{cm}) is strict, then $F$ is said to be strictly
completely monotonic on $I$. The classical Bernstein's theorem \cite%
{Bernstein-AM-52-1929}, \cite{Widder-TAMS-33-1931} states that a function $F$
is completely monotonic (for short, CM) on $(0,\infty )$ if and only if it
is a Laplace transform of some nonnegative measure $\mu $, that is,%
\begin{equation*}
F\left( x\right) =\int_{0}^{\infty }e^{-xt}d\mu \left( t\right) ,
\end{equation*}%
where $\mu \left( t\right) $ is non-decreasing and the integral converges
for $0<x<\infty $.

Another important one is the Bernstein functions \cite{Schilling-BFTA-2010}.
A non-negative function $F$ is said to be a Bernstein function on an
interval $I$, if $F$ has derivatives of all orders on $I$ and satisfies%
\begin{equation}
(-1)^{n}F^{(n)}(x)\leq 0\text{ for all }x\in I\text{ and }n=0,1,2,....
\label{bf}
\end{equation}%
Clearly, a function $F$ is a Bernstein function on $I$ if and only if $%
F^{\prime }$ is CM on $I$.

Very recently, Yang and Tian \cite{Yang-JIA-317-2017} established a
monotonicity rule for the ratio of two Laplace transforms as follows.

\begin{theorem}
\label{T-Rmr}Let the functions $f,g$ be defined on $\left( 0,\infty \right) $
such that their Laplace transforms $\mathcal{L}\left( f\right)
=\int_{0}^{\infty }f\left( t\right) e^{-xt}dt$ and $\mathcal{L}\left(
g\right) =\int_{0}^{\infty }g\left( t\right) e^{-xt}dt$ exist with $g\left(
t\right) \neq 0$ for all $t>0$. Then the ratio $\mathcal{L}\left( f\right) /%
\mathcal{L}\left( g\right) $ is decreasing (increasing) on $\left( 0,\infty
\right) $ if $f/g$ is increasing (decreasing) on $\left( 0,\infty \right) $.
\end{theorem}

By using this monotonicity rule, Yang and Tian proved that the function%
\begin{equation*}
x\mapsto \frac{1}{24x\left( \ln \Gamma \left( x+1/2\right) -x\ln x+x-\ln 
\sqrt{2\pi }\right) +1}-\frac{120}{7}x^{2}
\end{equation*}%
is strictly increasing from $\left( 0,\infty \right) $ onto $\left(
1,1860/343\right) $. In another paper \cite{Yang-JMAA-455-2017}, this
monotonicity rule was applied to investigate the monotonicity of the function%
\begin{equation*}
x\mapsto \frac{\psi ^{\left( n+1\right) }\left( x\right) ^{2}}{\psi ^{\left(
n\right) }\left( x\right) \psi ^{\left( n+2\right) }\left( x\right) }
\end{equation*}%
on $\left( 0,\infty \right) $, where $\psi ^{\left( n\right) }$ for $n\in 
\mathbb{N}$ is the polygamma functions, and obtained some new properties of
polygamma functions. These show that Theorem \ref{T-Rmr} is an efficient
tool of studying special functions.

Moreover, as shown in \cite[Remark 4]{Yang-JIA-317-2017}, if $g\left(
t\right) >0$ for all $t>0$, then by Theorem \ref{T-Rmr} and Bernstein's
theorem, both the functions%
\begin{equation*}
x\mapsto \mathcal{L}\left( f\right) -\beta \mathcal{L}\left( g\right) \text{
\ and \ }x\mapsto \alpha \mathcal{L}\left( g\right) -\mathcal{L}\left(
f\right)
\end{equation*}%
are CM on $\left( 0,\infty \right) $, where $\beta =\inf_{x>0}\left( 
\mathcal{L}\left( f\right) /\mathcal{L}\left( g\right) \right) >-\infty $
and $\alpha =\sup_{x>0}\left( \mathcal{L}\left( f\right) /\mathcal{L}\left(
g\right) \right) <\infty $.

Inspired by the above comments, the aim of this paper is to further
establish the monotonicity rule of the ratio of two Laplace transforms $%
\mathcal{L}\left( f\right) /\mathcal{L}\left( g\right) $ under the condition
that there is a $t^{\ast }>0$ such that $f/g$ is increasing (decreasing) on $%
\left( 0,t^{\ast }\right) $ and decreasing (increasing) on $\left( t^{\ast
},\infty \right) $.

The rest of this paper is organized as follows. In Section 2, some lemmas
are given, which containing monotonicity rules for ratios of two power
series (polynomials). In Section 3, our main results (Theorems 1--3) are
proved by means of definition of integral and lemmas. As applications, two
monotonicity results involving psi function and modified Bessel functions of
the second kind are presented.

\section{Lemmas}

To state needed lemmas, we recall a useful auxiliary function $H_{f,g}$,
which was introduced in \cite{Yang-arxiv-1409.6408}. For $-\infty \leq
a<b\leq \infty $, let $f$ and $g$ be differentiable functions on $(a,b)$
with $g^{\prime }\neq 0$ on $(a,b)$. Then we define%
\begin{equation}
H_{f,g}:=\frac{f^{\prime }}{g^{\prime }}g-f.  \label{H_f,g}
\end{equation}%
The auxiliary function $H_{f,g}$ has the following well properties \cite[%
Property 1]{Yang-arxiv-1409.6408}:

(i) $H_{f,g}$ is even with respect to $g$ and odd with respect to $f$, that
is,%
\begin{equation}
H_{f,g}\left( x\right) =H_{f,-g}\left( x\right) =-H_{-f,g}\left( x\right)
=-H_{-f,-g}\left( x\right) .  \label{H-sr}
\end{equation}

(ii) If $g\neq 0$ on $\left( a,b\right) $, then%
\begin{equation}
\left( \frac{f}{g}\right) ^{\prime }=\frac{g^{\prime }}{g^{2}}H_{f,g},
\label{df/g}
\end{equation}%
and therefore,%
\begin{equation}
\func{sgn}\left( \frac{f}{g}\right) ^{\prime }=\func{sgn}\left( g^{\prime
}\right) \func{sgn}\left( H_{f,g}\right) .  \label{sgnd(f/g)}
\end{equation}

(iii) If $f$ and $g$ are twice differentiable on $(a,b)$, then%
\begin{equation}
H_{f,g}^{\prime }=\left( \frac{f^{\prime }}{g^{\prime }}\right) ^{\prime }g.
\label{dH_f,g}
\end{equation}

The auxiliary function $H_{f,g}$ and its properties are very helpful to
investigate those monotonicity of ratios of two functions, see \cite%
{Yang-arxiv-1409.6408}, \cite{Yang-JIA-2016-221}, \cite{Yang-JIA-2016-251}, 
\cite{Yang-JIA-2016-311}, \cite{Lv-JIA-2017-94}, \cite{Yang-MIA-20(3)-2017}, 
\cite{Yang-MIA-20(4)-2017}, \cite{Luo-RM-72-2017}, \cite{Yang-JMI-11-2017}, 
\cite{Yang-JMI-12-2018}. Recently, in \cite{Yang-JMAA-428-2015} they were
successfully applied to establish monotonicity rules for ratios of two power
series and of two polynomials under the condition that the ratio of
coefficients of two power series is increasing (decreasing) then decreasing
(increasing). The following monotonicity rule for ration of two polynomials
will be used in proof of our main results.

\begin{lemma}[{\protect\cite[Theorem 2.5]{Yang-JMAA-428-2015}}]
\label{L-PA/PB-pm}Let $A_{n}\left( t\right) =\sum_{k=0}^{n}a_{k}t^{k}$ and $%
B_{n}\left( t\right) =\sum_{k=0}^{n}b_{k}t^{k}$ be two real polynomials
defined on $\left( 0,r\right) $ ($r>0$) with $b_{k}>0$ for all $0\leq k\leq
n $. Suppose that for certain $m\in \mathbb{N}$ with $m<n$, the sequences $%
\{a_{k}/b_{k}\}_{0\leq k\leq m}$ and $\{a_{k}/b_{k}\}_{m\leq k\leq n}$ are
both non-constants, and are respectively increasing (decreasing) and
decreasing (increasing). Then the function $A_{n}/B_{n}$ is increasing
(decreasing) on $\left( 0,r\right) $ if and only if $H_{A_{n},B_{n}}\left(
r^{-}\right) \geq \left( \leq \right) 0$. While $H_{A_{n},B_{n}}\left(
r^{-}\right) <\left( >\right) 0$, there is a unique $t_{0}\in \left(
0,r\right) $ such that the function $A_{n}/B_{n}$ is increasing (decreasing)
on $\left( 0,t_{0}\right) $ and decreasing (increasing) on $\left(
t_{0},r\right) $.
\end{lemma}

The following monotonicity rule will be used in Proposition \ref{P-p1},
which first appeared in \cite[Lemma 6.4]{Belzunce-I:M-E-40-2007} without
giving the details of the proof. Two strict proofs were given in \cite%
{Yang-JMAA-428-2015} and \cite{Xia-PJAM-7(2)-2016}.

\begin{lemma}[{\protect\cite[Corollary 2.6]{Yang-JMAA-428-2015}}]
\label{L-A/B-g}Let $A\left( t\right) =\sum_{k=0}^{\infty }a_{k}t^{k}$ and $%
B\left( t\right) =\sum_{k=0}^{\infty }b_{k}t^{k}$ be two real power series
converging on $\mathbb{R}$ with $b_{k}>0$ for all $k$. If for certain $m\in 
\mathbb{N}$, the non-constant sequences $\{a_{k}/b_{k}\}_{0\leq k\leq m}$
and $\{a_{k}/b_{k}\}_{k\geq m}$ are respectively increasing (decreasing) and
decreasing (increasing), then there is a unique $t_{0}\in \left( 0,\infty
\right) $ such that the function $A/B$ is increasing (decreasing) on $\left(
0,t_{0}\right) $ and decreasing (increasing) on $\left( t_{0},\infty \right) 
$.
\end{lemma}

The following lemma \cite[Lemma 2]{Yang-arxiv-1705-05704} offers a simple
criterion to determine the sign of a class of special series, which will be
used in proof of Proposition \ref{P-B1}.

\begin{lemma}[{\protect\cite[Lemma 2]{Yang-arxiv-1705-05704}}]
\label{L-sgnS}Let $\{a_{k}\}_{k=0}^{\infty }$ be a nonnegative real sequence
with $a_{m}>0$ and $\sum_{k=m+1}^{\infty }a_{k}>0$ and let%
\begin{equation*}
S\left( t\right) =-\sum_{k=0}^{m}a_{k}t^{k}+\sum_{k=m+1}^{\infty }a_{k}t^{k}
\end{equation*}%
be a convergent power series on the interval $\left( 0,r\right) $ ($r>0$).
(i) If $S\left( r^{-}\right) \leq 0$, then $S\left( t\right) <0$ for all $%
t\in \left( 0,r\right) $. (ii) If $S\left( r^{-}\right) >0$, then there is a
unique $t_{0}\in \left( 0,r\right) $ such that $S\left( t\right) <0$ for $%
t\in \left( 0,t_{0}\right) $ and $S\left( t\right) >0$ for $t\in \left(
t_{0},r\right) $.
\end{lemma}

\begin{remark}
Clearly, when $r=\infty $ in Lemma \ref{L-sgnS}, there is a unique $t_{0}\in
\left( 0,\infty \right) $ such that $S\left( t\right) <0$ for $t\in \left(
0,t_{0}\right) $ and $S\left( t\right) >0$ for $t\in \left( t_{0},\infty
\right) $. This result appeared in \cite[Lemma 6.3]{Belzunce-I:M-E-40-2007}
without proof (see also \cite{Yang-JIA-2015-157}, \cite{Yang-JIA-2015-299}).
If $a_{k}=0$ for $k\geq n\geq m+1$, then Lemma \ref{L-sgnS} is reduced to a
polynomial version, which appeared in \cite{Yang-AAA-2014-702718} (see also 
\cite{Yang-JMI-12-2018}).
\end{remark}

\section{Main results}

We are in a position to state and prove our main results.

\begin{theorem}
\label{MT-1}For $0<a<b<\infty $, let the functions $F$ and $G$ be defined on 
$\left( 0,\infty \right) $ by%
\begin{equation*}
F\left( x\right) =\int_{a}^{b}f\left( t\right) e^{-xt}dt\text{ \ and \ }%
G\left( x\right) =\int_{a}^{b}g\left( t\right) e^{-xt}dt,
\end{equation*}%
where the functions $f,g$ are both continuous on $\left[ a,b\right] $ with $%
g\left( t\right) >0$ for $t\in \left[ a,b\right] $. If there is $t^{\ast
}\in \left( a,b\right) $ such that $f/g$ is strictly increasing (decreasing)
on $\left[ a,t^{\ast }\right] $ and strictly decreasing (increasing) on $%
\left[ t^{\ast },b\right] $, then the ratio $x\mapsto F\left( x\right)
/G\left( x\right) $ is decreasing (increasing) on $\left( 0,\infty \right) $
if and only if%
\begin{equation*}
H_{F,G}\left( 0^{+}\right) =\lim_{x\rightarrow 0^{+}}\left( \frac{F^{\prime
}\left( x\right) }{G^{\prime }\left( x\right) }G\left( x\right) -F\left(
x\right) \right) \geq \left( \leq \right) 0,
\end{equation*}%
with%
\begin{equation}
\lim_{x\rightarrow 0^{+}}\frac{F\left( x\right) }{G\left( x\right) }=\frac{%
\int_{a}^{b}f\left( t\right) dt}{\int_{a}^{b}g\left( t\right) dt}\text{ \
and \ }\lim_{x\rightarrow \infty }\frac{F\left( x\right) }{G\left( x\right) }%
=\frac{f\left( a\right) }{g\left( a\right) }.  \label{F/G-0,00}
\end{equation}%
While $H_{F,G}\left( 0^{+}\right) <\left( >\right) 0$, then there is at
least one $x^{\ast }>0$ such that $F/G$ is increasing (decreasing) on $%
\left( 0,x^{\ast }\right) $ and decreasing (increasing) on $\left( x^{\ast
},\infty \right) $.
\end{theorem}

\begin{proof}[Proof of Theorem \protect\ref{MT-1}]
We only prove this theorem under the condition that $f/g$ is strictly
increasing on $\left[ a,t^{\ast }\right] $ and strictly decreasing on $\left[
t^{\ast },b\right] $. If $f/g$ is strictly decreasing on $\left[ a,t^{\ast }%
\right] $ and strictly increasing on $\left[ t^{\ast },b\right] $, then $%
\left( -f\right) /g$ is strictly increasing on $\left[ a,t^{\ast }\right] $
and strictly decreasing on $\left[ t^{\ast },b\right] $, and then
corresponding conclusion of this theorem is also true, which suffices to
note that $H_{-F,G}\left( x\right) =-H_{F,G}\left( x\right) $ due to (\ref%
{H-sr}).

For $n\in \mathbb{N}$, given a partition of the interval $\left[ a,b\right] $%
: 
\begin{equation*}
a=t_{0}<t_{1}<t_{2}<\cdot \cdot \cdot <t_{n}=b,
\end{equation*}%
with $\Delta t_{i}=t_{i}-t_{i-1}=\left( b-a\right) /n$, and so $%
t_{i}=a+\left( b-a\right) i/n$. Then we have%
\begin{eqnarray*}
\sum_{i=0}^{n-1}f\left( t_{i}\right) e^{-xt_{i}}\Delta t_{i}
&=&\sum_{i=0}^{n-1}f\left( t_{i}\right) e^{-ax}\left[ e^{-\left( b-a\right)
x/n}\right] ^{i}\frac{b-a}{n} \\
&=&\frac{b-a}{n}e^{-ax}\sum_{i=0}^{n-1}f\left( t_{i}\right) y^{i}=:\frac{b-a%
}{n}e^{-ax}F_{n}\left( y\right) ,
\end{eqnarray*}%
\begin{equation*}
\sum_{i=0}^{n-1}g\left( t_{i}\right) e^{-xt_{i}}\Delta t_{i}=\frac{b-a}{n}%
e^{-ax}\sum_{i=0}^{n-1}g\left( t_{i}\right) y^{i}:=\frac{b-a}{n}%
e^{-ax}G_{n}\left( y\right) ,
\end{equation*}%
where $y\equiv y_{n}\left( x\right) =e^{-\left( b-a\right) x/n}\in \left(
0,1\right) $. These imply that%
\begin{equation*}
\frac{\sum_{i=0}^{n-1}f\left( t_{i}\right) e^{-xt_{i}}\Delta t_{i}}{%
\sum_{i=0}^{n-1}g\left( t_{i}\right) e^{-xt_{i}}\Delta t_{i}}=\frac{%
\sum_{i=0}^{n-1}f\left( t_{i}\right) y^{i}}{\sum_{i=0}^{n-1}g\left(
t_{i}\right) y^{i}}=\frac{F_{n}\left( y\right) }{G_{n}\left( y\right) },
\end{equation*}%
\begin{equation}
\frac{F\left( x\right) }{G\left( x\right) }=\frac{\int_{a}^{b}f\left(
t\right) e^{-xt}dt}{\int_{a}^{b}g\left( t\right) e^{-xt}dt}=\frac{%
\lim_{n\rightarrow \infty }\sum_{i=0}^{n-1}f\left( t_{i}\right)
e^{-xt_{i}}\Delta t_{i}}{\lim_{n\rightarrow \infty }\sum_{i=0}^{n-1}g\left(
t_{i}\right) e^{-xt_{i}}\Delta t_{i}}=\lim_{n\rightarrow \infty }\frac{%
F_{n}\left( y\right) }{G_{n}\left( y\right) }.  \label{limFn/Gn}
\end{equation}%
Also, we have%
\begin{equation*}
\sum_{i=0}^{n-1}\left( t_{i}-a\right) f\left( t_{i}\right) e^{-xt_{i}}\Delta
t_{i}=\sum_{i=0}^{n-1}\frac{\left( b-a\right) i}{n}f\left( t_{i}\right)
e^{-ax}\left[ e^{-\left( b-a\right) x/n}\right] ^{i}\frac{b-a}{n}
\end{equation*}%
\begin{equation*}
=e^{-ax}\left( \frac{b-a}{n}\right) ^{2}y\sum_{i=0}^{n-1}if\left(
t_{i}\right) y^{i-1}=e^{-ax}\left( \frac{b-a}{n}\right) ^{2}yF_{n}^{\prime
}\left( y\right) ,
\end{equation*}%
\begin{equation*}
\sum_{i=0}^{n-1}\left( t_{i}-a\right) g\left( t_{i}\right) e^{-xt_{i}}\Delta
t_{i}=e^{-ax}\left( \frac{b-a}{n}\right) ^{2}yG_{n}^{\prime }\left( y\right)
.
\end{equation*}%
Therefore, we obtain%
\begin{equation*}
\frac{\sum\limits_{i=0}^{n-1}\left( t_{i}-a\right) f\left( t_{i}\right)
e^{-xt_{i}}\Delta t_{i}}{\sum\limits_{i=0}^{n-1}\left( t_{i}-a\right)
g\left( t_{i}\right) e^{-xt_{i}}\Delta t_{i}}=\frac{F_{n}^{\prime }\left(
y\right) }{G_{n}^{\prime }\left( y\right) },
\end{equation*}%
\begin{equation}
\frac{\int_{a}^{b}\left( t-a\right) f\left( t\right) e^{-xt}dt}{%
\int_{a}^{b}\left( t-a\right) g\left( t\right) e^{-xt}dt}=\frac{%
\lim_{n\rightarrow \infty }\sum\limits_{i=0}^{n-1}\left( t_{i}-a\right)
f\left( t_{i}\right) e^{-xt_{i}}\Delta t_{i}}{\lim_{n\rightarrow \infty
}\sum\limits_{i=0}^{n-1}\left( t_{i}-a\right) g\left( t_{i}\right)
e^{-xt_{i}}\Delta t_{i}}=\lim_{n\rightarrow \infty }\frac{F_{n}^{\prime
}\left( y\right) }{G_{n}^{\prime }\left( y\right) }.  \label{ldFn/dGn}
\end{equation}%
Since%
\begin{eqnarray*}
\int_{a}^{b}\left( t-a\right) f\left( t\right) e^{-xt}dt
&=&\int_{a}^{b}tf\left( t\right) e^{-xt}dt-a\int_{a}^{b}f\left( t\right)
e^{-xt}dt=-F^{\prime }\left( x\right) -aF\left( x\right) , \\
\int_{a}^{b}\left( t-a\right) g\left( t\right) e^{-xt}dt
&=&\int_{a}^{b}tg\left( t\right) e^{-xt}dt-a\int_{a}^{b}g\left( t\right)
e^{-xt}dt=-G^{\prime }\left( x\right) -aG\left( x\right) ,
\end{eqnarray*}%
equation (\ref{ldFn/dGn}) also can be written as%
\begin{equation}
\frac{F^{\prime }\left( x\right) +aF\left( x\right) }{G^{\prime }\left(
x\right) +aG\left( x\right) }=\lim_{n\rightarrow \infty }\frac{F_{n}^{\prime
}\left( y\right) }{G_{n}^{\prime }\left( y\right) }.  \label{ldFn/dGn-a}
\end{equation}%
It then follows that%
\begin{equation*}
H_{F_{n},G_{n}}\left( y\right) =\frac{F_{n}^{\prime }\left( y\right) }{%
G_{n}^{\prime }\left( y\right) }G_{n}\left( y\right) -F_{n}\left( y\right) =%
\frac{\sum\limits_{i=0}^{n-1}\left( t_{i}-a\right) f\left( t_{i}\right)
e^{-xt_{i}}\Delta t_{i}}{\sum\limits_{i=0}^{n-1}\left( t_{i}-a\right)
g\left( t_{i}\right) e^{-xt_{i}}\Delta t_{i}}\sum_{i=0}^{n-1}g\left(
t_{i}\right) y^{i}-\sum_{i=0}^{n-1}f\left( t_{i}\right) y^{i}
\end{equation*}%
\begin{eqnarray*}
&=&\frac{ne^{ax}}{b-a}\left( \frac{\sum_{i=0}^{n-1}\left( t_{i}-a\right)
f\left( t_{i}\right) e^{-xt_{i}}\Delta t_{i}}{\sum_{i=0}^{n-1}\left(
t_{i}-a\right) g\left( t_{i}\right) e^{-xt_{i}}\Delta t_{i}}\sum_{i=0}^{n-1}%
\frac{b-a}{n}e^{-ax}g\left( t_{i}\right) y^{i}-\sum_{i=0}^{n-1}\frac{b-a}{n}%
e^{-ax}f\left( t_{i}\right) y^{i}\right) \\
&=&\frac{ne^{ax}}{b-a}\left( \frac{\sum_{i=0}^{n-1}\left( t_{i}-a\right)
f\left( t_{i}\right) e^{-xt_{i}}\Delta t_{i}}{\sum_{i=0}^{n-1}\left(
t_{i}-a\right) g\left( t_{i}\right) e^{-xt_{i}}\Delta t_{i}}%
\sum_{i=0}^{n-1}g\left( t_{i}\right) e^{-xt_{i}}\Delta
t_{i}-\sum_{i=0}^{n-1}f\left( t_{i}\right) e^{-xt_{i}}\Delta t_{i}\right) \\
&:&=\frac{ne^{ax}}{b-a}C_{f,g}^{\left[ n\right] }\left( x\right) ,
\end{eqnarray*}%
and so%
\begin{equation}
\func{sgn}\left( H_{F_{n},G_{n}}\left( y\right) \right) =\func{sgn}\left(
C_{f,g}^{\left[ n\right] }\left( x\right) \right) .  \label{sgnHn-Cn}
\end{equation}%
Clearly, we have%
\begin{eqnarray*}
\lim_{n\rightarrow \infty }C_{f,g}^{\left[ n\right] }\left( x\right) &=&%
\frac{\int_{a}^{b}\left( t-a\right) f\left( t\right) e^{-xt}dt}{%
\int_{a}^{b}\left( t-a\right) g\left( t\right) e^{-xt}dt}\int_{a}^{b}g\left(
t\right) e^{-xt}dt-\int_{a}^{b}f\left( t\right) e^{-xt}dt \\
&=&\frac{F^{\prime }\left( x\right) +aF\left( x\right) }{G^{\prime }\left(
x\right) +aG\left( x\right) }G\left( x\right) -F\left( x\right) =\frac{%
G^{\prime }\left( x\right) }{G^{\prime }\left( x\right) +aG\left( x\right) }%
H_{F,G}\left( x\right) ,
\end{eqnarray*}%
which, in view of $G^{\prime }\left( x\right) <0$ and%
\begin{equation}
G^{\prime }\left( x\right) +aG\left( x\right) =-\int_{a}^{b}\left(
t-a\right) g\left( t\right) e^{-xt}dt<0  \label{dG+aG<0}
\end{equation}%
for $x\in \left( 0,\infty \right) $, implies that%
\begin{equation}
\func{sgn}\left( \lim_{n\rightarrow \infty }C_{f,g}^{\left[ n\right] }\left(
x\right) \right) =\func{sgn}\left( H_{F,G}\left( x\right) \right)
\label{sgnC-H}
\end{equation}%
for $x\in \left( 0,\infty \right) $.

(i) The necessity follows from%
\begin{equation*}
\left( \frac{F\left( x\right) }{G\left( x\right) }\right) ^{\prime }=\frac{%
G^{\prime }\left( x\right) }{G\left( x\right) ^{2}}H_{F,G}\left( x\right)
\leq 0
\end{equation*}%
for $x>0$, which, due to $G^{\prime }\left( x\right) <0$ for all $x\in
\left( 0,\infty \right) $, implies that $H_{F,G}\left( 0^{+}\right) \geq 0$.

Conversely, if $H_{F,G}\left( 0^{+}\right) \geq 0$, then by the relation (%
\ref{sgnC-H}), there is a large $N\in \mathbb{N}$ such that $C_{f,g}^{\left[
n\right] }\left( 0^{+}\right) \geq 0$ for $n>N$, which in combination with
the relation (\ref{sgnHn-Cn}) gives that $H_{F_{n},G_{n}}\left( y\right)
\geq 0$ as $y\rightarrow 1^{-}$ for $n>N$. On the other hand, since $f/g$ is
increasing on $\left[ a,t^{\ast }\right] $ and decreasing on $\left[ t^{\ast
},b\right] $, we easily see that there is a $i_{0}\geq 1$ such that the
sequence $\{f\left( t_{i}\right) /g\left( t_{i}\right) \}$ is strictly
increasing for $0\leq i\leq i_{0}$ and decreasing for $i_{0}<i\leq n-1$. By
Lemma \ref{L-PA/PB-pm}, the ratio $F_{n}\left( y\right) /G_{n}\left(
y\right) $ is strictly increasing with respect to $y$ on $\left( 0,1\right) $%
, that is,%
\begin{equation*}
\frac{d}{dy}\frac{F_{n}\left( y\right) }{G_{n}\left( y\right) }=\frac{%
G_{n}^{\prime }\left( y\right) }{G_{n}\left( y\right) }\left( \frac{%
F_{n}^{\prime }\left( y\right) }{G_{n}^{\prime }\left( y\right) }-\frac{%
F_{n}\left( y\right) }{G_{n}\left( y\right) }\right) >0\text{ for }n>N\text{
and }y\in \left( 0,1\right) ,
\end{equation*}%
which, due to $G_{n}\left( y\right) ,G_{n}^{\prime }\left( y\right) >0$,
yields%
\begin{equation*}
\frac{F_{n}^{\prime }\left( y\right) }{G_{n}^{\prime }\left( y\right) }-%
\frac{F_{n}\left( y\right) }{G_{n}\left( y\right) }>0\text{ for }n>N\text{
and }y\in \left( 0,1\right) .
\end{equation*}%
This together with (\ref{limFn/Gn}) and (\ref{ldFn/dGn-a}) gives%
\begin{equation*}
\frac{F^{\prime }\left( x\right) +aF\left( x\right) }{G^{\prime }\left(
x\right) +aG\left( x\right) }-\frac{F\left( x\right) }{G\left( x\right) }%
\geq 0\text{ for }x\in \left( 0,\infty \right) ,
\end{equation*}%
which indicates that%
\begin{equation*}
\left( \frac{F\left( x\right) }{G\left( x\right) }\right) ^{\prime }=\frac{%
G^{\prime }\left( x\right) +aG\left( x\right) }{G\left( x\right) }\left( 
\frac{F^{\prime }\left( x\right) +aF\left( x\right) }{G^{\prime }\left(
x\right) +aG\left( x\right) }-\frac{F\left( x\right) }{G\left( x\right) }%
\right) \leq 0\text{ for }x\in \left( 0,\infty \right) ,
\end{equation*}%
where the inequality holds due to $G\left( x\right) >0$ and $G^{\prime
}\left( x\right) +aG\left( x\right) <0$ by (\ref{dG+aG<0}). This proves the
sufficiency.

The first limit of (\ref{F/G-0,00}) is clear. While the second one follows
from (\ref{limFn/Gn}), which implies that%
\begin{eqnarray*}
\lim_{x\rightarrow \infty }\frac{F\left( x\right) }{G\left( x\right) }
&=&\lim_{x\rightarrow \infty }\lim_{n\rightarrow \infty }\frac{F_{n}\left(
y\right) }{G_{n}\left( y\right) }=\lim_{n\rightarrow \infty
}\lim_{x\rightarrow \infty }\frac{F_{n}\left( y\right) }{G_{n}\left(
y\right) } \\
&=&\lim_{n\rightarrow \infty }\lim_{y\rightarrow 0}\frac{F_{n}\left(
y\right) }{G_{n}\left( y\right) }=\frac{f\left( t_{0}\right) }{g\left(
t_{0}\right) }=\frac{f\left( a\right) }{g\left( a\right) }.
\end{eqnarray*}

(ii) If $H_{F,G}\left( 0^{+}\right) <0$, by the relations (\ref{sgnC-H}) and
(\ref{sgnHn-Cn}), there is a large enough $N\in \mathbb{N}$ such that $%
H_{F_{n},G_{n}}\left( y\right) <0$ as $y\rightarrow 1^{-}$ for $n>N$. By
Lemma \ref{L-PA/PB-pm}, there is a unique $y_{0}^{\left[ n\right] }\in
\left( 0,1\right) $ for given $n>N$ such that the function $F_{n}\left(
y\right) /G_{n}\left( y\right) $ is increasing on $\left( 0,y_{0}^{\left[ n%
\right] }\right) $ and decreasing on $\left( y_{0}^{\left[ n\right]
},1\right) $, that is,%
\begin{eqnarray*}
\frac{d}{dy}\frac{F_{n}\left( y\right) }{G_{n}\left( y\right) } &=&\frac{%
G_{n}^{\prime }\left( y\right) }{G_{n}\left( y\right) }\left( \frac{%
F_{n}^{\prime }\left( y\right) }{G_{n}^{\prime }\left( y\right) }-\frac{%
F_{n}\left( y\right) }{G_{n}\left( y\right) }\right) >0\text{ for }n>N\text{
and }y\in \left( 0,y_{0}^{\left[ n\right] }\right) , \\
\frac{d}{dy}\frac{F_{n}\left( y\right) }{G_{n}\left( y\right) } &=&\frac{%
G_{n}^{\prime }\left( y\right) }{G_{n}\left( y\right) }\left( \frac{%
F_{n}^{\prime }\left( y\right) }{G_{n}^{\prime }\left( y\right) }-\frac{%
F_{n}\left( y\right) }{G_{n}\left( y\right) }\right) <0\text{ for }n>N\text{
and }y\in \left( y_{0}^{\left[ n\right] },1\right) ,
\end{eqnarray*}%
where $y_{0}^{\left[ n\right] }$ is the unique solution of the equation $%
\left[ F_{n}\left( y\right) /G_{n}\left( y\right) \right] ^{\prime }=0$ on $%
\left( 0,1\right) $, namely,%
\begin{equation*}
\left[ \frac{d}{dy}\frac{F_{n}\left( y\right) }{G_{n}\left( y\right) }\right]
_{y=y_{0}^{\left[ n\right] }}=\frac{G_{n}^{\prime }\left( y_{0}^{\left[ n%
\right] }\right) }{G_{n}\left( y_{0}^{\left[ n\right] }\right) }\left( \frac{%
F_{n}^{\prime }\left( y_{0}^{\left[ n\right] }\right) }{G_{n}^{\prime
}\left( y_{0}^{\left[ n\right] }\right) }-\frac{F_{n}\left( y_{0}^{\left[ n%
\right] }\right) }{G_{n}\left( y_{0}^{\left[ n\right] }\right) }\right) =0%
\text{ for }n>N.
\end{equation*}

In the same treatment as part (i) of the proof of this theorem, the above
three relations imply that%
\begin{eqnarray*}
\left( \frac{F\left( x\right) }{G\left( x\right) }\right) ^{\prime } &\leq &0%
\text{ for }x\in \left( x^{\ast },\infty \right) , \\
\left( \frac{F\left( x\right) }{G\left( x\right) }\right) ^{\prime } &\geq &0%
\text{ for }x\in \left( 0,x^{\ast }\right) ,
\end{eqnarray*}%
where $x^{\ast }=\lim_{n\rightarrow \infty }x_{0}^{\left[ n\right] }$, $%
x_{0}^{\left[ n\right] }=-n\left( \ln y_{0}^{\left[ n\right] }\right)
/\left( b-a\right) $, and satisfies%
\begin{equation*}
\left[ \left( \frac{F\left( x\right) }{G\left( x\right) }\right) ^{\prime }%
\right] _{x=x^{\ast }}=0.
\end{equation*}%
Thus it remains to prove $x^{\ast }=\lim_{n\rightarrow \infty }x_{0}^{\left[
n\right] }\neq 0,\infty $. First, we claim that $x^{\ast
}=\lim_{n\rightarrow \infty }x_{0}^{\left[ n\right] }\neq 0$. If not, that
is, $\lim_{n\rightarrow \infty }x_{0}^{\left[ n\right] }=0$, then $F\left(
x\right) /G\left( x\right) $ is decreasing in $x$ on $\left( 0,\infty
\right) $. This, by part (i) of this theorem, implies that $H_{F,G}\left(
0^{+}\right) \geq 0$, which yields a contraction with the assumption that $%
H_{F,G}\left( 0^{+}\right) <0$.

Second, we also claim that $x^{\ast }=\lim_{n\rightarrow \infty }x_{0}^{%
\left[ n\right] }\neq \infty $. If not, that is, $\lim_{n\rightarrow \infty
}x_{0}^{\left[ n\right] }=\infty $, then $F\left( x\right) /G\left( x\right) 
$ is increasing in $x$ on $\left( 0,\infty \right) $. It then follows that
for all $x>0$,%
\begin{equation*}
\frac{F\left( x\right) }{G\left( x\right) }<\lim_{x\rightarrow \infty }\frac{%
F\left( x\right) }{G\left( x\right) }=\frac{f\left( a\right) }{g\left(
a\right) }.
\end{equation*}%
Since $\lim_{n\rightarrow \infty }\left( F_{n}\left( y\right) /G_{n}\left(
y\right) \right) =F\left( x\right) /G\left( x\right) $, there exists a large
enough $N_{1}\in \mathbb{N}$ such that for $n>N_{1}$ the inequality%
\begin{equation}
\frac{F_{n}\left( y\right) }{G_{n}\left( y\right) }<\frac{f\left( a\right) }{%
g\left( a\right) }  \label{Fn/Gn<}
\end{equation}%
holds for all $y\in \left( 0,1\right) $. On the other hand, as shown just
now, the function $F_{n}\left( y\right) /G_{n}\left( y\right) $ is
increasing on $\left( 0,y_{0}^{\left[ n\right] }\right) $ and decreasing on $%
\left( y_{0}^{\left[ n\right] },1\right) $, which suggests that there exists
a small enough $\delta \in \left( 0,y_{0}^{\left[ n\right] }\right) $ such
that $F_{n}\left( y\right) /G_{n}\left( y\right) $ is increasing on $\left(
0,\delta \right) $. Therefore, for $y\in \left( 0,\delta \right) $ and $n>N$,%
\begin{equation*}
\frac{F_{n}\left( y\right) }{G_{n}\left( y\right) }\geq \frac{F_{n}\left(
0\right) }{G_{n}\left( 0\right) }=\frac{f\left( t_{0}\right) }{g\left(
t_{0}\right) }=\frac{f\left( a\right) }{g\left( a\right) }.
\end{equation*}%
This is in contradiction with the inequality (\ref{Fn/Gn<}) for all $y\in
\left( 0,1\right) $ and $n>N_{1}$.

Consequently, $x^{\ast }=\lim_{n\rightarrow \infty }x_{0}^{\left[ n\right]
}\neq 0,\infty $, which ends the proof.
\end{proof}

Letting $a\rightarrow 0^{+}$, $b\rightarrow \infty $ in Theorem \ref{MT-1}.
Then $F\left( x\right) $ and $G\left( x\right) $ are Laplace transforms of
the functions $f$ and $g$, respectively. By properties of uniformly
convergent improper integral with a parameter, we have the following
monotonicity rule, where the first limit of (\ref{LF/G-0,00}) follows from
the first one of (\ref{F/G-0,00}) and Cauchy mean value theorem, that is,%
\begin{eqnarray*}
\lim_{x\rightarrow 0^{+}}\frac{F\left( x\right) }{G\left( x\right) }
&=&\lim_{b\rightarrow \infty }\frac{\int_{a}^{b}f\left( t\right) dt}{%
\int_{0}^{b}g\left( t\right) dt}=\lim_{b\rightarrow \infty }\frac{\left[
\int_{a}^{s}f\left( t\right) dt\right] _{s=b}-\left[ \int_{a}^{s}f\left(
t\right) dt\right] _{s=a}}{\left[ \int_{a}^{s}g\left( t\right) dt\right]
_{s=b}-\left[ \int_{0}^{s}g\left( t\right) dt\right] _{s=a}} \\
&=&\lim_{b\rightarrow \infty }\frac{f\left( a+\theta \left( b-a\right)
\right) }{g\left( a+\theta \left( b-a\right) \right) }=\lim_{t\rightarrow
\infty }\frac{f\left( t\right) }{g\left( t\right) },
\end{eqnarray*}%
here $\theta \in \left( 0,1\right) $.

\begin{theorem}
\label{MT-2}Let $f$ and $g$ be both continuous functions on $\left( 0,\infty
\right) $ with $g\left( t\right) >0$ for $t\in \left( 0,\infty \right) $ and
let $F\left( x\right) =\mathcal{L}\left( f\right) $ and $G\left( x\right) =%
\mathcal{L}\left( g\right) $ converge for $x>0$. If there is a $t^{\ast }\in
\left( 0,\infty \right) $ such that $f/g$ is strictly increasing
(decreasing) on $\left( 0,t^{\ast }\right) $ and strictly decreasing
(increasing) on $\left( t^{\ast },\infty \right) $, then the function $F/G$
is decreasing (increasing) on $\left( 0,\infty \right) $ if and only if%
\begin{equation*}
H_{F,G}\left( 0^{+}\right) =\lim_{x\rightarrow 0^{+}}\left( \frac{F^{\prime
}\left( x\right) }{G^{\prime }\left( x\right) }G\left( x\right) -F\left(
x\right) \right) \geq \left( \leq \right) 0,
\end{equation*}%
with%
\begin{equation}
\lim_{x\rightarrow 0^{+}}\frac{F\left( x\right) }{G\left( x\right) }%
=\lim_{t\rightarrow \infty }\frac{f\left( t\right) }{g\left( t\right) }\text{
\ and \ }\lim_{x\rightarrow \infty }\frac{F\left( x\right) }{G\left(
x\right) }=\lim_{t\rightarrow 0^{+}}\frac{f\left( t\right) }{g\left(
t\right) }  \label{LF/G-0,00}
\end{equation}%
provide the indicated limits exist. While $H_{F,G}\left( 0^{+}\right)
<\left( >\right) 0$, there is at leas one $x^{\ast }>0$ such that $F/G$ is
increasing (decreasing) on $\left( 0,x^{\ast }\right) $ and decreasing
(increasing) on $\left( x^{\ast },\infty \right) $.
\end{theorem}

\begin{theorem}
\label{MT-3}Suppose that (i) both the functions $f$ and $g$ are continuous
on $\left( 0,\infty \right) $ with $g\left( t\right) >0$ for $t\in \left(
0,\infty \right) $; (ii) the function $\mu $ is positive, differentiable and
increasing from $\left( 0,\infty \right) $ onto $\left( \mu \left(
0^{+}\right) ,\infty \right) $; (iii) both the functions%
\begin{equation*}
F\left( x\right) =\int_{0}^{\infty }f\left( t\right) e^{-x\mu \left(
t\right) }dt\text{ \ and \ }G\left( x\right) =\int_{0}^{\infty }g\left(
t\right) e^{-x\mu \left( t\right) }dt
\end{equation*}%
converge for all $x>0$. Then the following statements are valid:

(i) If the ratio $f/g$ is increasing (decreasing) on $\left( 0,\infty
\right) $, then $F/G$ is decreasing (increasing) on $\left( 0,\infty \right) 
$ with%
\begin{equation*}
\lim_{x\rightarrow 0^{+}}\frac{F\left( x\right) }{G\left( x\right) }%
=\lim_{t\rightarrow \infty }\frac{f\left( t\right) }{g\left( t\right) }\text{
\ and \ }\lim_{x\rightarrow \infty }\frac{F\left( x\right) }{G\left(
x\right) }=\lim_{t\rightarrow 0^{+}}\frac{f\left( t\right) }{g\left(
t\right) }.
\end{equation*}

(ii) If there is a $t^{\ast }\in \left( 0,\infty \right) $ such that $f/g$
is strictly increasing (decreasing) on $\left( 0,t^{\ast }\right) $ and
strictly decreasing (increasing) on $\left( t^{\ast },\infty \right) $, then
the ratio $F/G$ is decreasing (increasing) on $\left( 0,\infty \right) $ if
and only if%
\begin{equation*}
H_{F,G}\left( 0^{+}\right) =\lim_{x\rightarrow 0^{+}}\left( \frac{F^{\prime
}\left( x\right) }{G^{\prime }\left( x\right) }G\left( x\right) -F\left(
x\right) \right) \geq \left( \leq \right) 0.
\end{equation*}%
While $H_{F,G}\left( 0^{+}\right) <\left( >\right) 0$, there is at least one 
$x^{\ast }>0$ such that $F/G$ is increasing (decreasing) on $\left(
0,x^{\ast }\right) $ and decreasing (increasing) on $\left( x^{\ast },\infty
\right) $.
\end{theorem}

\begin{proof}
Let $\mu \left( t\right) -a=s$, where $a=\mu \left( 0^{+}\right) $. Then $%
F\left( x\right) $ and $G\left( x\right) $ are expressed as%
\begin{equation*}
F\left( x\right) =e^{-xa}\int_{0}^{\infty }\frac{f\left( t\left( s\right)
\right) }{\mu ^{\prime }\left( t\left( s\right) \right) }e^{-xs}ds\text{ \
and \ }G\left( x\right) =e^{-xa}\int_{0}^{\infty }\frac{g\left( t\left(
s\right) \right) }{\mu ^{\prime }\left( t\left( s\right) \right) }e^{-xs}ds,
\end{equation*}%
where $t\left( s\right) =\mu ^{-1}\left( s+a\right) $, and $F\left( x\right)
/G\left( x\right) $ can be represented in the form of ratio of two Laplace
transforms:%
\begin{equation*}
\frac{F\left( x\right) }{G\left( x\right) }=\frac{e^{xa}F\left( x\right) }{%
e^{xa}G\left( x\right) }=\frac{\int_{0}^{\infty }\left[ f\left( t\left(
s\right) \right) /\mu ^{\prime }\left( t\left( s\right) \right) \right]
e^{-xs}ds}{\int_{0}^{\infty }\left[ g\left( t\left( s\right) \right) /\mu
^{\prime }\left( t\left( s\right) \right) \right] e^{-xs}ds}:=\frac{%
\int_{0}^{\infty }f^{\ast }\left( s\right) e^{-xs}ds}{\int_{0}^{\infty
}g^{\ast }\left( s\right) e^{-xs}ds}.
\end{equation*}%
It is easy to verify that%
\begin{equation}
\frac{f^{\ast }\left( s\right) }{g^{\ast }\left( s\right) }=\frac{f\left(
t\left( s\right) \right) }{g\left( t\left( s\right) \right) }\text{, \ \ \ }%
\left( \frac{f^{\ast }\left( s\right) }{g^{\ast }\left( s\right) }\right)
^{\prime }=\left( \frac{f\left( t\right) }{g\left( t\right) }\right)
^{\prime }\times \frac{dt}{ds}=\frac{1}{\mu ^{\prime }\left( t\right) }%
\left( \frac{f\left( t\right) }{g\left( t\right) }\right) ^{\prime },
\label{f*/g*-f/g}
\end{equation}%
where $\mu ^{\prime }\left( t\right) >0$ for all $t>0$.

(i) If the ratio $f\left( t\right) /g\left( t\right) $ is increasing
(decreasing) on $\left( 0,\infty \right) $, then so is $f^{\ast }\left(
s\right) /g^{\ast }\left( s\right) $. By Theorem \ref{T-Rmr}, we easily find
that $\left( e^{xa}F\right) /\left( e^{xa}G\right) =F/G$ is decreasing
(increasing) on $\left( 0,\infty \right) $.

By the limit relations (\ref{LF/G-0,00}) we easily get%
\begin{eqnarray*}
\lim_{x\rightarrow 0^{+}}\frac{F\left( x\right) }{G\left( x\right) }
&=&\lim_{s\rightarrow \infty }\frac{f\left( t\left( s\right) \right) /\mu
^{\prime }\left( t\left( s\right) \right) }{g\left( t\left( s\right) \right)
/\mu ^{\prime }\left( t\left( s\right) \right) }=\lim_{t\rightarrow \infty }%
\frac{f\left( t\right) }{g\left( t\right) }, \\
\lim_{x\rightarrow \infty }\frac{F\left( x\right) }{G\left( x\right) }
&=&\lim_{s\rightarrow 0^{+}}\frac{f\left( t\left( s\right) \right) /\mu
^{\prime }\left( t\left( s\right) \right) }{g\left( t\left( s\right) \right)
/\mu ^{\prime }\left( t\left( s\right) \right) }=\lim_{t\rightarrow 0^{+}}%
\frac{f\left( t\right) }{g\left( t\right) }.
\end{eqnarray*}

(ii) If there is a $t^{\ast }\in \left( 0,\infty \right) $ such that $%
f\left( t\right) /g\left( t\right) $ is strictly increasing (decreasing) on $%
\left( 0,t^{\ast }\right) $ and strictly decreasing (increasing) on $\left(
t^{\ast },\infty \right) $, then by the second relation of (\ref{f*/g*-f/g}%
), there is a $s^{\ast }\in \left( 0,\infty \right) $ such that $f^{\ast
}\left( s\right) /g^{\ast }\left( s\right) $ is strictly increasing
(decreasing) on $\left( 0,s^{\ast }\right) $ and strictly decreasing
(increasing) on $\left( s^{\ast },\infty \right) $, where $s^{\ast }=\mu
\left( t^{\ast }\right) -a$. By Theorem \ref{MT-2}, the ratio $%
e^{xa}F/\left( e^{xa}G\right) =F/G$ is decreasing (increasing) on $\left(
0,\infty \right) $ if and only if%
\begin{equation}
\lim_{x\rightarrow 0^{+}}H_{e^{xa}F,e^{xa}G}\left( x\right)
=\lim_{x\rightarrow 0^{+}}\left( \frac{\left( e^{xa}F\left( x\right) \right)
^{\prime }}{\left( e^{xa}G\left( x\right) \right) ^{\prime }}e^{xa}G\left(
x\right) -e^{xa}F\left( x\right) \right) \geq \left( \leq \right) 0.
\label{HF*,G*}
\end{equation}%
We claim that the limit relation is equivalent to $\lim_{x\rightarrow
0^{+}}H_{F,G}\left( x\right) \geq \left( \leq \right) 0$. In fact, we easily
check that%
\begin{eqnarray*}
H_{e^{xa}F,e^{xa}G}\left( x\right) &=&\frac{e^{xa}\left( F^{\prime }\left(
x\right) +aF\left( x\right) \right) }{e^{xa}\left( G^{\prime }\left(
x\right) +aG\left( x\right) \right) }e^{xa}G\left( x\right) -e^{xa}F\left(
x\right) \\
&=&e^{xa}\frac{F^{\prime }\left( x\right) G\left( x\right) +aF\left(
x\right) G\left( x\right) -F\left( x\right) G^{\prime }\left( x\right)
-aF\left( x\right) G\left( x\right) }{G^{\prime }\left( x\right) +aG\left(
x\right) } \\
&=&\frac{e^{xa}G^{\prime }\left( x\right) }{G^{\prime }\left( x\right)
+aG\left( x\right) }\left( \frac{F^{\prime }\left( x\right) }{G^{\prime
}\left( x\right) }G\left( x\right) -F\left( x\right) \right) =\frac{%
e^{xa}G^{\prime }\left( x\right) }{G^{\prime }\left( x\right) +aG\left(
x\right) }H_{F,G}\left( x\right) .
\end{eqnarray*}%
This together with%
\begin{eqnarray*}
G^{\prime }\left( x\right) &=&-\int_{0}^{\infty }\mu \left( t\right) g\left(
t\right) e^{-x\mu \left( t\right) }dt<0, \\
G^{\prime }\left( x\right) +aG\left( x\right) &=&-\int_{0}^{\infty }\left[
\mu \left( t\right) -\mu \left( 0^{+}\right) \right] g\left( t\right)
e^{-x\mu \left( t\right) }dt<0
\end{eqnarray*}%
for $x>0$ indicates that%
\begin{equation*}
\func{sgn}\left( H_{e^{xa}F,e^{xa}G}\left( x\right) \right) =\func{sgn}%
\left( H_{F,G}\left( x\right) \right) ,
\end{equation*}%
which proves the claim just now.

(iii) If $\lim_{x\rightarrow 0^{+}}H_{F,G}\left( x\right) <\left( >\right) 0$%
, then $\lim_{x\rightarrow 0^{+}}H_{e^{xa}F,e^{xa}G}\left( x\right) <\left(
>\right) 0$. By Theorem \ref{MT-2}, there is at least one $x^{\ast }>0$ such
that $e^{xa}F\left( x\right) /\left( e^{xa}G\left( x\right) \right) =F\left(
x\right) /G\left( x\right) $ is increasing (decreasing) on $\left( 0,x^{\ast
}\right) $ and decreasing (increasing) on $\left( x^{\ast },\infty \right) $.

Thus we complete the proof.
\end{proof}

\section{A unified treatment for certain bounds of harmonic number}

The Euler-Mascheroni constant is defined by%
\begin{equation*}
\gamma =\lim_{n\rightarrow \infty }\left( H_{n}-\ln n\right) =0.577215664...,
\end{equation*}%
where $H_{n}=\sum_{k=1}^{n}k^{-1}$ is the $n$'th harmonic number. There is a
close connection between $H_{n}$ and the psi (or digamma) function. Indeed,
we have $H_{n}=\psi \left( n+1\right) +\gamma $. Several bounds for $H_{n}$
or $\psi \left( n+1\right) $ can see \cite{Tims-MG-55-1971}, \cite%
{Young-MG-75-1991}, \cite{DeTemple-AMM-100-1993}, \cite{Negoi-GM-15-1997}, 
\cite{Alzer-AMSUH-68-1998}, \cite{Villarino-arXiv-0510585}, \cite%
{Villarino-JIPAM-9(3)-2008}, \cite{Mortici-CMA-59-2010}, \cite%
{Chen-AML-23-2010}, \cite{Qi-AMC-218-2011}, \cite{Chen-CMA-64-2013}, \cite%
{Lu-JMAA-419-2014}, \cite{Yang-AMC-268-2015} \cite{Zhao-JIA-193-2015}.

In particular, Alzer \cite{Alzer-AMSUH-68-1998} obtained the double
inequality, 
\begin{equation*}
\frac{1}{2\left( n+a\right) }\leq H_{n}-\ln n-\gamma <\frac{1}{2\left(
n+b\right) }
\end{equation*}%
holds for $n\in \mathbb{N}$ with the best constants%
\begin{equation*}
a=\frac{1}{2\left( 1-\gamma \right) }-1\text{ \ and \ }b=\frac{1}{6}
\end{equation*}%
by proving the sequence%
\begin{equation}
A\left( n\right) =\frac{1}{2}\frac{1}{\psi \left( n+1\right) -\ln n}-n
\label{A(n)}
\end{equation}%
is strictly decreasing for $n\geq 1$.

Villarino \cite{Villarino-arXiv-0510585} showed that%
\begin{eqnarray}
H_{n} &=&\ln \sqrt{n\left( n+1\right) }+\gamma +\frac{1}{6n\left( n+1\right)
+L\left( n\right) },  \label{L(n)} \\
&=&\ln \left( n+\frac{1}{2}\right) +\gamma +\frac{1}{24\left( n+1/2\right)
^{2}+D\left( n\right) },  \label{D(n)}
\end{eqnarray}%
where both the sequences $L\left( n\right) $ and $D\left( n\right) $ are
increasing for $n\in \mathbb{N}$. Qi \cite{Qi-AMC-218-2011} showed that the
sequence%
\begin{equation}
Q\left( n\right) =\frac{1}{2}\frac{1}{\ln n+1/\left( 2n\right) -\psi \left(
n+1\right) }-12n^{2}  \label{Q(n)}
\end{equation}%
is strictly increasing for $n\in N$. These monotonicity of sequences $%
L\left( n\right) $, $D\left( n\right) $ and $Q\left( n\right) $ similarly
yield corresponding sharp bounds for $H_{n}$ or $\psi \left( n+1\right) $.

We remark that it is difficult to deal with the monotonicity of the function 
$A\left( x\right) $, $L\left( x\right) $, $D\left( x\right) $ and $Q\left(
x\right) $ on $\left( 0,\infty \right) $ by usual approach. However, if we
write them as ratios of two Laplace transforms, then we easily prove their
monotonicity on $\left( 0,\infty \right) $ by Theorems \ref{T-Rmr}\ and \ref%
{MT-2}. Here we chose $\Phi \left( x\right) =D\left( x-1/2\right) $ defined
by (\ref{D(n)}) and prove its monotonicity on $\left( 0,\infty \right) $. As
far as $A\left( x\right) $, $L\left( x\right) $ and $Q\left( x\right) $, we
only list their expressions in the form of ratios of Laplace transforms. In
fact, By means of the formulas%
\begin{eqnarray*}
\psi (x) &=&\int_{0}^{\infty }\left( \frac{e^{-t}}{t}-\frac{e^{-xt}}{1-e^{-t}%
}\right) dt, \\
\ln x &=&\int_{0}^{\infty }\frac{e^{-t}-e^{-xt}}{t}dt\text{, } \\
\frac{1}{x^{n}} &=&\frac{1}{\left( n-1\right) !}\int_{0}^{\infty
}t^{n-1}e^{-xt}dt,
\end{eqnarray*}%
we have%
\begin{equation*}
A\left( x\right) =\frac{1}{2}\frac{1}{\psi \left( x+1\right) -\ln x}-x=\frac{%
\int_{0}^{\infty }\left[ -p_{1}^{\prime }\left( t\right) \right] e^{-xt}dt}{%
\int_{0}^{\infty }p\left( t\right) e^{-xt}dt},
\end{equation*}%
where%
\begin{equation*}
p_{1}\left( t\right) =2\left( \frac{1}{t}-\frac{1}{e^{t}-1}\right) ;
\end{equation*}%
\begin{equation*}
L\left( x\right) =\frac{2}{2\psi \left( x+1\right) -\ln \left( x\left(
x+1\right) \right) }-6x\left( x+1\right) =\frac{\int_{0}^{\infty }\left[
-6\left( p_{2}^{\prime \prime }\left( t\right) +p_{2}^{\prime }\left(
t\right) \right) \right] e^{-xt}dt}{\int_{0}^{\infty }p_{2}\left( t\right)
e^{-xt}dt},
\end{equation*}%
where%
\begin{equation*}
p_{2}\left( t\right) =\frac{e^{2t}-2te^{t}-1}{t\left( e^{t}-1\right) e^{t}};
\end{equation*}%
\begin{equation*}
Q\left( x\right) =\frac{1}{\ln x+1/\left( 2x\right) -\psi \left( x+1\right) }%
-12x^{2}=\frac{\int_{0}^{\infty }\left[ -12p_{3}^{\prime \prime }\left(
t\right) \right] e^{-xt}dt}{\int_{0}^{\infty }p_{3}\left( t\right) e^{-xt}dt}%
,
\end{equation*}%
where%
\begin{equation*}
p_{3}\left( t\right) =\frac{1}{2}\left( \coth \frac{t}{2}-\frac{2}{t}\right)
.
\end{equation*}%
Next we prove the monotonicity of $\Phi \left( x\right) =D\left(
x-1/2\right) $ on $\left( 0,\infty \right) $ by Theorem \ref{MT-2}.

\begin{proposition}
\label{P-p1}The function%
\begin{equation*}
\Phi \left( x\right) =\frac{1}{\psi \left( x+1/2\right) -\ln x}-24x^{2}
\end{equation*}%
is strictly increasing from $\left( 0,\infty \right) $ onto $\left(
0,21/5\right) $.
\end{proposition}

\begin{proof}
We write $\Phi \left( x\right) =F\left( x\right) /G\left( x\right) $, where%
\begin{eqnarray*}
F\left( x\right) &=&1-24x^{2}\left( \psi \left( x+1/2\right) -\ln x\right) ,
\\
G\left( x\right) &=&\psi \left( x+1/2\right) -\ln x.
\end{eqnarray*}%
It has been shown in \cite{Yang-JIA-157-2015} that%
\begin{eqnarray*}
G\left( x\right) &=&\int_{0}^{\infty }q\left( t\right) e^{-xt}dt, \\
x^{2}G\left( x\right) &=&\frac{1}{24}+\int_{0}^{\infty }e^{-xt}q^{\prime
\prime }\left( t\right) dt,
\end{eqnarray*}%
where%
\begin{equation}
q\left( t\right) =\frac{1}{t}-\frac{1}{2\sinh \left( t/2\right) }.  \label{q}
\end{equation}%
Then $F\left( x\right) $ can be written as 
\begin{equation*}
F\left( x\right) =1-24x^{2}G\left( x\right) =\int_{0}^{\infty }\left(
-24q^{\prime \prime }\left( t\right) \right) e^{-xt}dt.
\end{equation*}%
Thus $\Phi \left( x\right) $ is expressed as%
\begin{equation*}
\Phi \left( x\right) =\frac{F\left( x\right) }{G\left( x\right) }=\frac{%
\int_{0}^{\infty }\left[ -24q^{\prime \prime }\left( t\right) \right]
e^{-xt}dt}{\int_{0}^{\infty }q\left( t\right) e^{-xt}dt}.
\end{equation*}%
We first show that there is a $t^{\ast }>0$ such that the function $%
-24q^{\prime \prime }/q$ is decreasing on $\left( 0,t^{\ast }\right) $ and
increasing on $\left( t^{\ast },\infty \right) $. Direct computations give%
\begin{eqnarray*}
\frac{q^{\prime \prime }\left( t\right) }{q\left( t\right) } &=&\frac{1}{4}%
\frac{2\cosh ^{2}s\sinh s-s^{3}\cosh ^{2}s-2\sinh s-s^{3}}{s^{2}\left(
s-\sinh s\right) \sinh ^{2}s} \\
&=&\frac{1}{2}\frac{\sinh 3s-s^{3}\cosh 2s-3\sinh s-3s^{3}}{s^{2}\left(
\sinh 3s-2s\cosh 2s-3\sinh s+2s\right) },
\end{eqnarray*}%
where $s=t/2$. Expanding in power series yields%
\begin{equation*}
\frac{q^{\prime \prime }\left( t\right) }{q\left( t\right) }=\frac{1}{2}%
\frac{\sum_{n=3}^{\infty }\frac{3^{2n-1}-\left( 2n-1\right) \left(
2n-2\right) \left( 2n-3\right) 2^{2n-4}-3}{\left( 2n-1\right) !}s^{2n-1}}{%
\sum_{n=2}^{\infty }\frac{3^{2n-3}-\left( 2n-3\right) 2^{2n-3}-3}{\left(
2n-3\right) !}s^{2n-1}}:=\frac{1}{2}\frac{\sum_{n=4}^{\infty }a_{n}\left(
s^{2}\right) ^{n-4}}{\sum_{n=4}^{\infty }b_{n}\left( s^{2}\right) ^{n-4}}.
\end{equation*}%
Since%
\begin{equation*}
\left( 2n-1\right) !b_{n+1}-9\left( 2n-3\right) !b_{n}=\left( 10n-23\right)
2^{2n-3}+24>0\text{ for }n\geq 3
\end{equation*}%
and $b_{3}=0$, we see that $b_{n}>0$ for $n\geq 4$.\ Thus, to prove the
function $-24q^{\prime \prime }/q$ is decreasing on $\left( 0,t^{\ast
}\right) $ and increasing on $\left( t^{\ast },\infty \right) $, by Lemma %
\ref{L-A/B-g} it suffices to prove that there is an integer $n_{0}>4$ such
that the sequence $\{a_{n}/b_{n}\}_{n\geq 4}$ is increasing for $4\leq n\leq
n_{0}$ and decreasing for $n>n_{0}$. For this end, we have to prove $d_{n}=%
\left[ \left( 2n-1\right) !\right] ^{2}\left(
a_{n}b_{n+1}-a_{n+1}b_{n}\right) <0$ for $4\leq n\leq 9$ and $d_{n}>0$ for $%
n\geq 10$.

Some elementary computations gives%
\begin{eqnarray*}
d_{n} &=&18\left( 4n-1\right) +\frac{2}{9}\left( 4n-1\right) 3^{4n}-\frac{1}{%
108}\left( 2n-1\right) \left( 20n^{4}-56n^{3}-77n^{2}+464n-243\right) 6^{2n}
\\
&&+\frac{4}{9}\left( 64n^{2}-132n+41\right) 3^{2n}-\frac{3}{4}\left(
2n-1\right) \left( 2n-3\right) \left( 6n^{3}+5n^{2}+2n-1\right) 2^{2n}.
\end{eqnarray*}%
We find that%
\begin{equation*}
\begin{array}{lll}
d_{4}=-66\,802\,176, & d_{5}=-13\,774\,616\,064, & d_{6}=-1570\,251\,361%
\,536, \\ 
d_{7}=-127\,269\,822\,161\,664, & d_{8}=-7526\,731\,991\,528\,448, & 
d_{9}=-240\,861\,038\,835\,686\,400.%
\end{array}%
\end{equation*}%
To prove $d_{n}>0$ for $n\geq 10$, we write $d_{n}$ as%
\begin{equation*}
d_{n}=18\left( 4n-1\right) +\left( 4n-1\right) 6^{2n}\times a_{n}^{\ast
}+\left( 64n^{2}-132n+41\right) 2^{2n}\times b_{n}^{\ast },
\end{equation*}%
where%
\begin{eqnarray*}
a_{n}^{\ast } &=&\frac{2}{9}\left( \frac{3}{2}\right) ^{2n}-\frac{1}{108}%
\frac{\left( 2n-1\right) \left( 20n^{4}-56n^{3}-77n^{2}+464n-243\right) }{%
4n-1}, \\
b_{n}^{\ast } &=&\frac{4}{9}\left( \frac{3}{2}\right) ^{2n}-\frac{3}{4}\frac{%
\left( 2n-1\right) \left( 2n-3\right) \left( 6n^{3}+5n^{2}+2n-1\right) }{%
64n^{2}-132n+41}.
\end{eqnarray*}%
It is easy to verify that%
\begin{equation*}
a_{10}^{\ast }=\frac{710\,697\,141}{6815\,744}>0\text{, \ \ \ }b_{10}^{\ast
}=\frac{174\,443\,916\,097}{149\,159\,936}>0,
\end{equation*}%
and for $n\geq 10$,%
\begin{equation*}
a_{n+1}^{\ast }-\frac{9}{4}a_{n}^{\ast }=\frac{1}{432}\frac{\left(
2n^{2}+n-9\right) \left( 400n^{4}-2500n^{3}+1448n^{2}+1789n-777\right) }{%
\left( 4n-1\right) \left( 4n+3\right) }>0,
\end{equation*}%
\begin{equation*}
\begin{array}{l}
b_{n+1}^{\ast }-\dfrac{9}{4}b_{n}^{\ast }=\dfrac{3}{16}\dfrac{2n-1}{\left(
64n^{2}-4n-27\right) \left( 64n^{2}-132n+41\right) }\bigskip \\ 
\times \left( 4n^{4}\left( 960n^{2}-3004n-1181\right)
+19\,204n^{3}+16\,517n^{2}-684n-2697\right) >0,%
\end{array}%
\end{equation*}%
which yield $d_{n}>0$ for $n\geq 10$.

Second, it is easy to see that%
\begin{equation*}
\lim_{x\rightarrow 0^{+}}F\left( x\right) =1\text{, \ \ \ }%
\lim_{x\rightarrow 0^{+}}F^{\prime }\left( x\right) =0,
\end{equation*}%
and%
\begin{equation*}
\lim_{x\rightarrow 0^{+}}\frac{G\left( x\right) }{G^{\prime }\left( x\right) 
}=\lim_{x\rightarrow 0^{+}}\frac{\psi \left( x+1/2\right) -\ln x}{\psi
^{\prime }\left( x+1/2\right) -1/x}=0,
\end{equation*}%
which yield%
\begin{equation*}
\lim_{x\rightarrow 0^{+}}H_{F,G}\left( x\right) =\lim_{x\rightarrow
0^{+}}\left( \frac{F^{\prime }\left( x\right) }{G^{\prime }\left( x\right) }%
G\left( x\right) -F\left( x\right) \right) =-1<0.
\end{equation*}%
It then follows by Theorem \ref{MT-2} that the function $\Phi =F/G$ is
strictly increasing on $\left( 0,\infty \right) $.

An easy computation gives%
\begin{equation*}
\lim_{x\rightarrow 0^{+}}\Phi \left( x\right) =\lim_{t\rightarrow \infty }%
\frac{-24q^{\prime \prime }\left( t\right) }{q\left( t\right) }=0\text{ \
and \ }\lim_{x\rightarrow \infty }\Phi \left( x\right) =\lim_{t\rightarrow 0}%
\frac{-24q^{\prime \prime }\left( t\right) }{q\left( t\right) }=\frac{21}{5},
\end{equation*}%
which completes the proof.
\end{proof}

\section{An application to Bessel functions}

The modified Bessel functions of the second kind $K_{v}$ is defined as \cite[%
p. 78]{Watson-ATTBF-CUP-1922}%
\begin{equation}
K_{v}\left( x\right) =\frac{\pi }{2}\frac{I_{-v}\left( x\right) -I_{v}\left(
x\right) }{\sin \left( v\pi \right) },  \label{Kv-Iv}
\end{equation}%
where $I_{v}\left( x\right) $ is the modified Bessel functions of the first
kind which can be represented by the infinite series as 
\begin{equation}
I_{v}\left( x\right) =\sum_{n=0}^{\infty }\frac{\left( x/2\right) ^{2n+v}}{%
n!\Gamma \left( v+n+1\right) }\text{, \ }x\in \mathbb{R}\text{, \ }v\in 
\mathbb{R}\backslash \{-1,-2,...\},  \label{I_v-is}
\end{equation}
and the right-hand side of (\ref{Kv-Iv}) is replaced by its limiting value
if $v$ is an integer or zero.

We easily see that $K_{v}\left( x\right) =K_{-v}\left( x\right) $ for all $%
v\in \mathbb{R}$ and $x>0$ by (\ref{Kv-Iv}), so we assume that $v\geq 0$ in
this section, unless otherwise specified.

As showed in proof of Theorem 3.1 in \cite{Baricz-BAMS-82-2010}, the identity%
\begin{equation}
1-\frac{K_{v-1}\left( x\right) K_{v+1}\left( x\right) }{K_{v}\left( x\right)
^{2}}=\frac{1}{x}\left( \frac{xK_{v}^{\prime }\left( x\right) }{K_{v}\left(
x\right) }\right) ^{\prime }  \label{T-xdK/K}
\end{equation}%
holds for $v\in \mathbb{R}$ and $x>0$, and the function%
\begin{equation*}
x\mapsto \frac{xK_{v}^{\prime }\left( x\right) }{K_{v}\left( x\right) }
\end{equation*}%
is strictly decreasing on $\left( 0,\infty \right) $ for all $v\in \mathbb{R}
$ (see also \cite{Yang-PAMS-145-2017}). As another application, in this
section we will determine the monotonicity of the function $x\mapsto
x+xK_{v}^{\prime }\left( x\right) /K_{v}\left( x\right) $ on $\left(
0,\infty \right) $ by Theorem \ref{MT-3}. More precisely, we have

\begin{proposition}
\label{P-B1}For $v\geq 0$, let $K_{v}\left( x\right) $ be the modified
Bessel functions of the second kind.

(i) If $v\in \left( 1/2,\infty \right) $, then the function%
\begin{equation*}
x\mapsto \Lambda \left( x\right) =x+\frac{xK_{v}^{\prime }\left( x\right) }{%
K_{v}\left( x\right) }
\end{equation*}%
is strictly increasing from $\left( 0,\infty \right) $ onto $\left(
-v,-1/2\right) $.

(ii) If $v\in \lbrack 0,1/2)$, then the function $x\mapsto \Lambda \left(
x\right) $ is strictly decreasing from $\left( 0,\infty \right) $ onto $%
\left( -1/2,-v\right) $.

(iii) If $v\neq 1/2$, then the double inequality%
\begin{equation}
-x-\max \left( v,\frac{1}{2}\right) <\frac{xK_{v}^{\prime }\left( x\right) }{%
K_{v}\left( x\right) }<-x-\min \left( v,\frac{1}{2}\right)  \label{xdK/K<>}
\end{equation}%
holds for $x>0$ with the best constants $\min \left( v,1/2\right) $ and $%
\max \left( v,1/2\right) $.
\end{proposition}

Before proving Proposition \ref{P-B1}, we give the following lemmas.

\begin{lemma}
\label{L-hv-m}For $v>0$ with $v\neq 1/2$, let the function $h_{v}$ be
defined on $\left( 0,\infty \right) $ by%
\begin{equation}
h_{v}\left( t\right) =\frac{\cosh \left( tv\right) +v\sinh \left( tv\right)
\sinh t}{\left( \cosh t+1\right) \cosh \left( tv\right) }.  \label{hv}
\end{equation}%
(i) If $v\in \lbrack 1,\infty )$ then $h_{v}$ is increasing from $\left(
0,\infty \right) $ onto $\left( 1/2,v\right) $.

(ii) If $v\in \left( 1/2,1\right) $, then there is a $t^{\ast }\in \left(
0,\infty \right) $ such that $h_{v}$ is increasing on $\left( 0,t^{\ast
}\right) $ and decreasing on $\left( t^{\ast },\infty \right) $.
Consequently, the inequalities%
\begin{equation*}
\frac{1}{2}=\min \left( \frac{1}{2},v\right) <h_{v}\left( t\right) \leq
\theta _{v}
\end{equation*}%
hold for $x>0$, where $\theta _{v}=h_{v}\left( t^{\ast }\right) $, here $%
t^{\ast }$ is the unique solution of the equation $h_{v}^{\prime }\left(
t\right) =0$ on $\left( 0,\infty \right) $.

(iii) If $v\in \left( 0,1/2\right) $, then there is a $t^{\ast }\in \left(
0,\infty \right) $ such that $h_{v}$ is decreasing on $\left( 0,t^{\ast
}\right) $ and increasing on $\left( t^{\ast },\infty \right) $. Therefore,
it holds that for $x>0$,%
\begin{equation*}
\theta _{v}\leq h_{v}\left( t\right) <\max \left( \frac{1}{2},v\right) =%
\frac{1}{2},
\end{equation*}%
where $\theta _{v}$ is as in (ii).
\end{lemma}

\begin{proof}
Differentiating and simplifying yield%
\begin{equation*}
h_{v}^{\prime }\left( t\right) =\frac{r_{v}\left( t\right) }{\left( \cosh
t+1\right) ^{2}\cosh ^{2}\left( tv\right) },
\end{equation*}%
where%
\begin{equation*}
r_{v}\left( t\right) =v^{2}\left( 1+\cosh t\right) \sinh t+v\left( 1+\cosh
t\right) \cosh \left( tv\right) \sinh \left( tv\right) -\cosh ^{2}\left(
tv\right) \sinh t.
\end{equation*}%
Using "product into sum" formulas and expanding in power series gives%
\begin{eqnarray*}
4r_{v}\left( t\right) &=&\left( v+1\right) \sinh \left( 2vt-t\right) +\left(
v-1\right) \sinh \left( 2tv+t\right) \\
&&+2v\sinh \left( 2tv\right) +2v^{2}\sinh \left( 2t\right) +2\left(
2v^{2}-1\right) \sinh t \\
&:&=\sum_{n=1}^{\infty }\frac{a_{n}}{\left( 2n-1\right) !}t^{2n-1},
\end{eqnarray*}%
where%
\begin{equation*}
a_{n}=\left( v+1\right) \left( 2v-1\right) ^{2n-1}+\left( v-1\right) \left(
2v+1\right) ^{2n-1}+\left( 2v\right) ^{2n}+v^{2}2^{2n}+2\left(
2v^{2}-1\right) .
\end{equation*}

To confirm the monotonicity of $h_{v}$, we need to deal with the sign of $%
a_{n}$. For this end, we first give two recurrence relations. It is easy to
check that%
\begin{equation}
\begin{array}{c}
\dfrac{a_{n+1}-a_{n}}{2^{2n}}=2v\left( v^{2}-1\right) \left( v-\frac{1}{2}%
\right) ^{2n-1}+2v\left( v^{2}-1\right) \left( v+\frac{1}{2}\right)
^{2n-1}\bigskip \\ 
+\left( 4v^{2}-1\right) v^{2n}+3v^{2}:=b_{n},%
\end{array}
\label{Dan}
\end{equation}%
\begin{equation}
\frac{b_{n+1}-b_{n}}{v\left( v^{2}-1\right) \left( v^{2}-1/4\right) }=\left(
2v-3\right) \left( v-\frac{1}{2}\right) ^{2n-2}+\left( 2v+3\right) \left( v+%
\frac{1}{2}\right) ^{2n-2}+4v^{2n-1}>0  \label{Dbn}
\end{equation}%
for $v\neq 1,1/2$. We now distinguish three cases to prove the desired
monotonicity.

\textbf{Case 1}: $v\geq 1$. It is clear that $a_{n}>0$ for $n\geq 1$, which
implies that $h_{v}^{\prime }\left( t\right) >0$ for $t\in \left( 0,\infty
\right) $.

\textbf{Case 2}: $v\in \left( 1/2,1\right) $. From the recurrence relation (%
\ref{Dbn}) we see that the sequence $\{b_{n}\}_{n\geq 1}$ is decreasing,
which together with%
\begin{eqnarray*}
b_{1} &=&2v^{2}\left( 2v-1\right) \left( 2v+1\right) >0, \\
\lim_{n\rightarrow \infty }b_{n} &=&\func{sgn}\left( v^{2}-1\right) \infty <0
\end{eqnarray*}%
yields that there is an integer $n_{1}>1$ such that $b_{n}\geq 0$ for $1\leq
n\leq n_{1}$ and $b_{n}\leq 0$ for $n\geq n_{1}$. This, by the recurrence
relation (\ref{Dan}), in turn implies that the sequence $\{a_{n}\}_{n\geq 1}$
is increasing for $1\leq n\leq n_{1}$ and decreasing for $n\geq n_{1}$.
Therefore, we obtain%
\begin{equation*}
a_{n}\geq a_{1}=4\left( 2v-1\right) \left( 2v+1\right) >0\text{ for }1\leq
n\leq n_{1}.
\end{equation*}%
On the other hand, it is seen that 
\begin{equation*}
\lim_{n\rightarrow \infty }\frac{a_{n}}{2^{2n}}=\func{sgn}\left( v-1\right)
\infty <0.
\end{equation*}%
It then follows that there is an integer $n_{0}>n_{1}$ such that $a_{n}\geq
0 $ for $1\leq n\leq n_{0}$ and $a_{n}\leq 0$ for $n\geq n_{0}$. By Lemma %
\ref{L-sgnS}, there is a $t^{\ast }>0$ such that $h_{v}^{\prime }\left(
t\right) >0$ for $t\in \left( 0,t^{\ast }\right) $ and $h_{v}^{\prime
}\left( t\right) <0$ for $t\in \left( t^{\ast },\infty \right) $.

\textbf{Case 3}: $v\in \left( 0,1/2\right) $. In this case, we see that the
sequence $\{b_{n}\}_{n\geq 1}$ is increasing, which in combination with%
\begin{eqnarray*}
b_{1} &=&2v^{2}\left( 2v-1\right) \left( 2v+1\right) <0, \\
\lim_{n\rightarrow \infty }b_{n} &=&3v^{2}>0
\end{eqnarray*}%
indicates that there is an integer $n_{1}>1$ such that $b_{n}\leq 0$ for $%
1\leq n\leq n_{1}$ and $b_{n}\geq 0$ for $n\geq n_{1}$. This, by the
recurrence relation (\ref{Dan}), means that the sequence $\{a_{n}\}_{n\geq
1} $ is decreasing for $1\leq n\leq n_{1}$ and increasing for $n\geq n_{1}$.
Hence, we deduce that%
\begin{equation*}
a_{n}\leq a_{1}=4\left( 2v-1\right) \left( 2v+1\right) <0\text{ for }1\leq
n\leq n_{1}.
\end{equation*}%
Moreover, it is seen that 
\begin{equation*}
\lim_{n\rightarrow \infty }\frac{a_{n}}{\left( 2v+1\right) ^{2n}}=\func{sgn}%
\left( v^{2}\right) \infty >0.
\end{equation*}%
It then follows that there is an integer $n_{0}>n_{1}$ such that $a_{n}\leq
0 $ for $1\leq n\leq n_{0}$ and $a_{n}\geq 0$ for $n\geq n_{0}$. By Lemma %
\ref{L-sgnS}, there is a $t^{\ast }>0$ such that $h_{v}^{\prime }\left(
t\right) <0$ for $t\in \left( 0,t^{\ast }\right) $ and $h_{v}^{\prime
}\left( t\right) >0$ for $t\in \left( t^{\ast },\infty \right) $.

This completes the proof.
\end{proof}

\begin{lemma}
\label{L-HF,G}Let $K_{v}\left( x\right) $ be the modified Bessel functions
of the second kind and let%
\begin{equation*}
F\left( x\right) =x\left[ K_{v}\left( x\right) +K_{v}^{\prime }\left(
x\right) \right] \text{ \ and \ }G\left( x\right) =K_{v}\left( x\right) .
\end{equation*}%
Then for $v\in \left( 0,1\right) $, we have%
\begin{equation*}
\lim_{x\rightarrow 0^{+}}H_{F,G}\left( x\right) =\lim_{x\rightarrow
0^{+}}\left( \frac{F^{\prime }\left( x\right) }{G^{\prime }\left( x\right) }%
G\left( x\right) -F\left( x\right) \right) =0.
\end{equation*}
\end{lemma}

\begin{proof}
By the asymptotic formulas \cite[p. 375, (9.6.9)]{Abramowitz-HMFFGMT-1972}%
\begin{equation}
K_{v}\left( x\right) \thicksim \frac{1}{2}\Gamma \left( v\right) \left( 
\frac{x}{2}\right) ^{-v}\text{ for\ }v>0\text{ as }x\rightarrow 0\text{,}
\label{K_v-->0}
\end{equation}%
we have that$\quad $as $x\rightarrow 0$, 
\begin{eqnarray*}
K_{v}^{\prime }\left( x\right) &\thicksim &-\frac{1}{4}\Gamma \left(
v+1\right) \left( \frac{x}{2}\right) ^{-v-1}\text{ for\ }v>0, \\
K_{v}^{\prime \prime }\left( x\right) &\thicksim &\frac{1}{8}\Gamma \left(
v+2\right) \left( \frac{x}{2}\right) ^{-v-2}\text{ for\ }v>0.
\end{eqnarray*}%
Then, for $v>0$, as $x\rightarrow 0^{+}$,%
\begin{eqnarray*}
F\left( x\right) &=&x\left[ K_{v}\left( x\right) +K_{v}^{\prime }\left(
x\right) \right] \thicksim x\left[ \frac{1}{2}\Gamma \left( v\right) \left( 
\frac{x}{2}\right) ^{-v}-\frac{1}{4}\Gamma \left( v+1\right) \left( \frac{x}{%
2}\right) ^{-v-1}\right] \\
&=&\Gamma \left( v\right) \frac{x-v}{2}\left( \frac{x}{2}\right) ^{-v},
\end{eqnarray*}%
\begin{eqnarray*}
F^{\prime }\left( x\right) &=&K_{v}\left( x\right) +\left( x+1\right)
K_{v}^{\prime }\left( x\right) +xK_{v}^{\prime \prime }\left( x\right) \\
&\thicksim &\frac{1}{2}\Gamma \left( v\right) \left( \frac{x}{2}\right)
^{-v}-\frac{x+1}{4}\Gamma \left( v+1\right) \left( \frac{x}{2}\right)
^{-v-1}+x\frac{1}{8}\Gamma \left( v+2\right) \left( \frac{x}{2}\right)
^{-v-2} \\
&=&\frac{1}{2}\Gamma \left( v\right) \frac{\left( 1-v\right) x+v^{2}}{x}%
\left( \frac{x}{2}\right) ^{-v},
\end{eqnarray*}%
which yield%
\begin{eqnarray*}
H_{F,G}\left( x\right) &=&\frac{F^{\prime }\left( x\right) }{G^{\prime
}\left( x\right) }G\left( x\right) -F\left( x\right) \\
&\thicksim &\frac{\frac{1}{2}\Gamma \left( v\right) \frac{\left( 1-v\right)
x+v^{2}}{x}\left( \frac{x}{2}\right) ^{-v}}{-\frac{1}{4}\Gamma \left(
v+1\right) \left( \frac{x}{2}\right) ^{-v-1}}\frac{1}{2}\Gamma \left(
v\right) \left( \frac{x}{2}\right) ^{-v}-\Gamma \left( v\right) \frac{x-v}{2}%
\left( \frac{x}{2}\right) ^{-v} \\
&=&-\frac{\Gamma \left( v\right) }{v}\left( \frac{x}{2}\right)
^{1-v}\rightarrow 0\text{ as }x\rightarrow 0^{+}\text{ if }v\in \left(
0,1\right) .
\end{eqnarray*}%
This completes the proof.
\end{proof}

We now are in a position to prove Proposition \ref{P-B1}.

\begin{proof}[Proof of Proposition \protect\ref{P-B1}]
We have 
\begin{equation*}
\Lambda \left( x\right) =\frac{x\left[ K_{v}\left( x\right) +K_{v}^{\prime
}\left( x\right) \right] }{K_{v}\left( x\right) }=\frac{F\left( x\right) }{%
G\left( x\right) }.
\end{equation*}%
By the integral representation (\ref{Kv-Ir}) we get that%
\begin{equation}
G\left( x\right) =K_{v}\left( x\right) =\int_{0}^{\infty }\cosh \left(
vt\right) e^{-x\cosh t}dt;  \label{G-ir}
\end{equation}%
\begin{eqnarray*}
K_{v}\left( x\right) +K_{v}^{\prime }\left( x\right) &=&\int_{0}^{\infty
}\cosh \left( vt\right) e^{-x\cosh t}dt-\int_{0}^{\infty }\cosh \left(
vt\right) \left( \cosh t\right) e^{-x\cosh t}dt \\
&=&-\int_{0}^{\infty }\cosh \left( vt\right) \left( \cosh t-1\right)
e^{-x\cosh t}dt,
\end{eqnarray*}%
then integration by parts yields%
\begin{eqnarray}
F\left( x\right) &=&x\left( K_{v}\left( x\right) +K_{v}^{\prime }\left(
x\right) \right) =\int_{0}^{\infty }\cosh \left( vt\right) \frac{\cosh t-1}{%
\sinh t}de^{-x\cosh t}  \notag \\
&=&-\int_{0}^{\infty }\frac{\cosh \left( vt\right) +v\sinh t\sinh \left(
vt\right) }{\cosh t+1}e^{-x\cosh t}dt.  \label{F-ir}
\end{eqnarray}%
Thus $\Lambda \left( x\right) $ can be written as%
\begin{equation*}
\Lambda \left( x\right) =\frac{F\left( x\right) }{G\left( x\right) }=\frac{%
\int_{0}^{\infty }\left( -\frac{\cosh \left( vt\right) +v\sinh t\sinh \left(
vt\right) }{\cosh t+1}\right) e^{-x\cosh t}dt}{\int_{0}^{\infty }\cosh
\left( vt\right) e^{-x\cosh t}dt}:=\frac{\int_{0}^{\infty }f\left( t\right)
e^{-x\mu \left( t\right) }dt}{\int_{0}^{\infty }g\left( t\right) e^{-x\mu
\left( t\right) }dt},
\end{equation*}%
where $\mu \left( t\right) =\cosh t$ and%
\begin{equation*}
f\left( t\right) =-\frac{\cosh \left( vt\right) +v\sinh t\sinh \left(
vt\right) }{\cosh t+1}\text{, \ \ \ }g\left( t\right) =\cosh \left(
vt\right) .
\end{equation*}%
Clearly, $\mu \left( t\right) $ is positive, differentiable and increasing
on $\left( 0,\infty \right) $, while $f\left( t\right) $ and $g\left(
t\right) $ are differentiable on $\left( 0,\infty \right) $ with $g\left(
t\right) >0$ for $t>0$. Also, we have%
\begin{equation*}
\frac{f\left( t\right) }{g\left( t\right) }=\frac{-\left[ \cosh \left(
vt\right) +v\sinh t\sinh \left( vt\right) \right] /\left( \cosh t+1\right) }{%
\cosh \left( vt\right) }=-h_{v}\left( t\right) .
\end{equation*}

(i) If $v\geq 1$, then by the first assertion of Lemma \ref{L-hv-m} we see
that $f/g$ is strictly decreasing on $\left( 0,\infty \right) $. It follows
by part (i) of Theorem \ref{MT-3} that $\Lambda =F/G$ is strictly increasing
on $\left( 0,\infty \right) $.

(ii) If $v\in \left( 1/2,1\right) $, then by the second assertion of Lemma %
\ref{L-hv-m} we see that there is a $t^{\ast }\in \left( 0,\infty \right) $
such that $f/g$ is decreasing on $\left( 0,t^{\ast }\right) $ and increasing
on $\left( t^{\ast },\infty \right) $. Also, $\lim_{x\rightarrow
0^{+}}H_{F,G}\left( x\right) =0$ due to Lemma \ref{L-HF,G}. These, by part
(ii) of Theorem \ref{MT-3}, yield that $\Lambda =F/G$ is strictly increasing
on $\left( 0,\infty \right) $.

(iii) If $v\in \left( 0,1/2\right) $, then the third assertion of Lemma \ref%
{L-hv-m} it is seen that there is a $t^{\ast }\in \left( 0,\infty \right) $
such that $f/g$ is increasing on $\left( 0,t^{\ast }\right) $ and decreasing
on $\left( t^{\ast },\infty \right) $. And, $\lim_{x\rightarrow
0^{+}}H_{F,G}\left( x\right) =0$ due to Lemma \ref{L-HF,G}. It then follows
from part (ii) of Theorem \ref{MT-3} that $\Lambda =F/G$ is strictly
decreasing on $\left( 0,\infty \right) $.

(iv) If $v=0$, then $f\left( t\right) /g\left( t\right) =-1/\left( \cosh
t+1\right) $ is clearly increasing on $\left( 0,\infty \right) $. It follows
from part (i) of Theorem \ref{MT-3} that $\Lambda =F/G$ is strictly
decreasing on $\left( 0,\infty \right) $.

The limit values are%
\begin{eqnarray*}
\Lambda \left( 0^{+}\right) &=&\lim_{x\rightarrow 0^{+}}\frac{F\left(
x\right) }{G\left( x\right) }=-\lim_{t\rightarrow \infty }\frac{\cosh \left(
vt\right) +v\sinh t\sinh \left( vt\right) }{\left( \cosh t+1\right) \cosh
\left( vt\right) }=-v, \\
\Lambda \left( \infty \right) &=&\lim_{x\rightarrow \infty }\frac{F\left(
x\right) }{G\left( x\right) }=-\lim_{t\rightarrow 0^{+}}\frac{\cosh \left(
vt\right) +v\sinh t\sinh \left( vt\right) }{\left( \cosh t+1\right) \cosh
\left( vt\right) }=-\frac{1}{2}.
\end{eqnarray*}%
The double inequality (\ref{xdK/K<>}) follows from the monotonicity of $F/G$
on $\left( 0,\infty \right) $, which ends the proof.
\end{proof}

As a consequence of \cite[Theorem 5]{Miller-ITSF-12(4)-2001}, Miller and
Samko showed that the function $x\mapsto \sqrt{x}e^{x}K_{v}\left( x\right) $
is strictly decreasing on $\left( 0,\infty \right) $ for $\left\vert
v\right\vert >1/2$. Yang and Zheng in \cite[Corollary 3.2]%
{Yang-PAMS-145-2017} reproved this assertion and further proved this
function is strictly increasing on $\left( 0,\infty \right) $ for $%
\left\vert v\right\vert <1/2$. Now we have a more general result by
Proposition \ref{P-B1}.

\begin{corollary}
\label{C-B1}For $v\geq 0$, let $K_{v}\left( x\right) $ be the modified
Bessel functions of the second kind. Then the function%
\begin{equation*}
x\mapsto x^{r}e^{x}K_{v}\left( x\right)
\end{equation*}%
is strictly increasing on $\left( 0,\infty \right) $ if and only if $r\geq
\max \left( v,1/2\right) $, and decreasing if and only if $r\leq \min \left(
v,1/2\right) $. While $v>\left( <\right) 1/2$ and $1/2<r<v$ ($v<r<1/2$),
there is an $x_{0}>0$ such that this function is decreasing (increasing) on $%
\left( 0,x_{0}\right) $ and increasing (decreasing) on $\left( x_{0},\infty
\right) $.
\end{corollary}

\begin{proof}
Differentiation yields%
\begin{eqnarray*}
\left[ x^{r}e^{x}K_{v}\left( x\right) \right] ^{\prime }
&=&x^{r}e^{x}K_{v}\left( x\right) +rx^{r-1}e^{x}K_{v}\left( x\right)
+x^{r}e^{x}K_{v}^{\prime }\left( x\right) \\
&=&x^{r-1}e^{x}K_{v}\left( x\right) \left( r+x+\frac{xK_{v}^{\prime }\left(
x\right) }{K_{v}\left( x\right) }\right) :=x^{r-1}e^{x}K_{v}\left( x\right)
\times \phi \left( x\right) .
\end{eqnarray*}%
By Proposition \ref{P-B1}, $\left[ x^{r}e^{x}K_{v}\left( x\right) \right]
^{\prime }\geq \left( \leq \right) 0$ for all $x>0$ if and only if%
\begin{eqnarray*}
r &\geq &-\inf_{x>0}\left( x+\frac{xK_{v}^{\prime }\left( x\right) }{%
K_{v}\left( x\right) }\right) =\left\{ 
\begin{array}{cc}
v & \text{if }v\in \left( 1/2,\infty \right) \\ 
\frac{1}{2} & \text{if }v\in \left( 0,1/2\right)%
\end{array}%
\right. =\max \left( v,\frac{1}{2}\right) , \\
r &\leq &-\sup_{x>0}\left( x+\frac{xK_{v}^{\prime }\left( x\right) }{%
K_{v}\left( x\right) }\right) =\left\{ 
\begin{array}{cc}
\frac{1}{2} & \text{if }v\in \left( 1/2,\infty \right) \\ 
v & \text{if }v\in \left( 0,1/2\right)%
\end{array}%
\right. =\min \left( v,\frac{1}{2}\right) .
\end{eqnarray*}%
When $v>1/2$ and $1/2<r<v$, by part (i) of Proposition \ref{P-B1}, we see
that $x\mapsto r+x+xK_{v}^{\prime }\left( x\right) /K_{v}\left( x\right)
=\phi \left( x\right) $ is increasing on $\left( 0,\infty \right) $, which
in combination with $\phi \left( 0^{+}\right) =r-v<0$ and $\phi \left(
\infty \right) =r-1/2>0$ yields that there is an $x_{0}>0$ such that $\phi
\left( x\right) <0$ for $x\in \left( 0,x_{0}\right) $ and $\phi \left(
x\right) >0$ for $x\in \left( x_{0},\infty \right) $. That is to say, the
function $x\mapsto x^{r}e^{x}K_{v}\left( x\right) $ is decreasing on $\left(
0,x_{0}\right) $ and increasing on $\left( x_{0},\infty \right) $.

Similarly, by part (ii) of Proposition \ref{P-B1} we can prove that for $%
v\in \lbrack 0,1/2)$ and $v<r<1/2$, there is an $x_{0}>0$ such that the
function $x\mapsto x^{r}e^{x}K_{v}\left( x\right) $ is increasing on $\left(
0,x_{0}\right) $ and decreasing on $\left( x_{0},\infty \right) $. This
completes the proof.
\end{proof}

\begin{remark}
Corollary \ref{C-B1} implies that for all $0<x<y$, the double inequality%
\begin{equation}
e^{y-x}\left( \frac{y}{x}\right) ^{r_{1}}<\frac{K_{v}\left( x\right) }{%
K_{v}\left( y\right) }<e^{y-x}\left( \frac{y}{x}\right) ^{r_{2}}
\label{Kx/Ky<>}
\end{equation}%
if and only if $r_{1}\leq \min \left( \left\vert v\right\vert ,1/2\right) $
and $r_{2}\geq \max \left( \left\vert v\right\vert ,1/2\right) $. Obviously,
our inequalities (\ref{Kx/Ky<>}) are superior to those earlier results
appeared in \cite{Bordelon-SIAMREV-15-1973}, \cite{Ross-SIAMREV-15-1973}, 
\cite{Paris-SIAM-JMA-15(1)-1984}, \cite{Joshi-JAMS-50-1991}, \cite%
{Laforgia-JVAM-34(3)-1991}. Detailed comments can see \cite[Section 3]%
{Baricz-PEMS-53-2010}.
\end{remark}

By Lemma \ref{L-hv-m} and Bernstein's theorem, we give a class of completely
monotonic function.

\begin{corollary}
\label{C-B2}For $v\geq 0$, the function%
\begin{equation*}
P_{\lambda }\left( x\right) =\left( x+\lambda \right) e^{x}K_{v}\left(
x\right) +xe^{x}K_{v}^{\prime }\left( x\right)
\end{equation*}%
is CM on $\left( 0,\infty \right) $ if and only if%
\begin{equation*}
\lambda \geq \left\{ 
\begin{array}{cc}
v & \text{if }v\in \lbrack 1,\infty ), \\ 
\theta _{v} & \text{if }v\in \left( \frac{1}{2},1\right) , \\ 
\frac{1}{2} & \text{if }v\in \left( 0,\frac{1}{2}\right) ,%
\end{array}%
\right.
\end{equation*}%
and so is $-P_{\lambda }\left( x\right) $ if and only if 
\begin{equation*}
\lambda \leq \left\{ 
\begin{array}{cc}
\frac{1}{2} & \text{if }v\in \left( \frac{1}{2},\infty \right) , \\ 
\theta _{v} & \text{if }v\in \left( 0,\frac{1}{2}\right) ,%
\end{array}%
\right.
\end{equation*}%
where $\theta _{v}$ is as in Lemma \ref{L-hv-m}.
\end{corollary}

\begin{proof}
From integral representations (\ref{F-ir}) and (\ref{G-ir}), we have%
\begin{eqnarray*}
P_{\lambda }\left( x\right)  &=&e^{x}F\left( x\right) +\lambda e^{x}G\left(
x\right)  \\
&=&\int_{0}^{\infty }\left( -\frac{\cosh \left( vt\right) +v\sinh t\sinh
\left( vt\right) }{\cosh t+1}\right) e^{-x\left( \cosh t-1\right)
}dt+\lambda \int_{0}^{\infty }\cosh \left( vt\right) e^{-x\left( \cosh
t-1\right) }dt
\end{eqnarray*}%
\begin{equation*}
=\int_{0}^{\infty }\left( \lambda -h_{v}\left( t\right) \right) \cosh \left(
vt\right) e^{-x\left( \cosh t-1\right) }dt=\int_{0}^{\infty }\left( \lambda
-h_{v}\left( t\left( s\right) \right) \right) \frac{\cosh \left( vt\left(
s\right) \right) }{\sinh \left( t\left( s\right) \right) }e^{-xs}ds,
\end{equation*}%
where $h_{v}\left( t\right) $ is defined by (\ref{hv}) and $t\left( s\right)
=\cosh ^{-1}\left( s+1\right) $. By Bernstein's theorem and Lemma \ref%
{L-hv-m} $P_{\lambda }\left( x\right) $ is CM on $\left( 0,\infty \right) $
if and only if%
\begin{equation*}
\lambda \geq \sup_{t>0}\left( h_{v}\left( t\right) \right) =\left\{ 
\begin{array}{cc}
v & \text{if }v\in \lbrack 1,\infty ), \\ 
\theta _{v} & \text{if }v\in \left( \frac{1}{2},1\right) , \\ 
\frac{1}{2} & \text{if }v\in \left( 0,\frac{1}{2}\right) ,%
\end{array}%
\right. 
\end{equation*}%
and so is $-P_{\lambda }\left( x\right) $ if and only if%
\begin{equation*}
\lambda \leq \inf_{t>0}\left( h_{v}\left( t\right) \right) =\left\{ 
\begin{array}{cc}
\frac{1}{2} & \text{if }v\in \lbrack 1,\infty ), \\ 
\frac{1}{2} & \text{if }v\in \left( \frac{1}{2},1\right) , \\ 
\theta _{v} & \text{if }v\in \left( 0,\frac{1}{2}\right) ,%
\end{array}%
\right. 
\end{equation*}%
which completes the proof.
\end{proof}

\begin{remark}
It was proved in \cite[Theorem 5]{Miller-ITSF-12(4)-2001} that the function $%
x\mapsto \sqrt{x}e^{x}K_{v}\left( x\right) $ is CM on $\left( 0,\infty
\right) $ if $\left\vert v\right\vert >1/2$. Now we present a new proof by
Corollary \ref{C-B2}. Differentiation yields%
\begin{equation*}
-\left[ \sqrt{x}e^{x}K_{v}\left( x\right) \right] ^{\prime }=-\frac{1}{\sqrt{%
x}}\left[ \left( x+\frac{1}{2}\right) e^{x}K_{v}\left( x\right)
+xe^{x}K_{v}^{\prime }\left( x\right) \right] =\frac{1}{\sqrt{x}}\left[
-P_{1/2}\left( x\right) \right] .
\end{equation*}%
Since $x\mapsto 1/\sqrt{x}$ and $-P_{1/2}\left( x\right) $ are CM on $\left(
0,\infty \right) $ by Corollary \ref{C-B2}, we find that so is $x^{-1/2}%
\left[ -P_{1/2}\left( x\right) \right] $, and so is $\sqrt{x}%
e^{x}K_{v}\left( x\right) $ on $\left( 0,\infty \right) $.
\end{remark}

\begin{remark}
Baricz \cite{Baricz-PEMS-53-2010} conjectured that $x\mapsto \sqrt{x}%
e^{x}K_{v}\left( x\right) $ for all $|v|<1/2$ is a Bernstein function, which
was proved in \cite[Remark 3.3]{Yang-PAMS-145-2017}. By Corollary \ref{C-B2}%
, we can give a simple proof. Indeed, it suffices to prove $\left[ \sqrt{x}%
e^{x}K_{v}\left( x\right) \right] ^{\prime }$ for $v\in \left( 0,1/2\right) $
is CM on $\left( 0,\infty \right) $. Differentiation yields%
\begin{equation*}
\left[ \sqrt{x}e^{x}K_{v}\left( x\right) \right] ^{\prime }=\frac{1}{\sqrt{x}%
}\left[ \left( x+\frac{1}{2}\right) e^{x}K_{v}\left( x\right)
+xe^{x}K_{v}^{\prime }\left( x\right) \right] =\frac{1}{\sqrt{x}}%
P_{1/2}\left( x\right) .
\end{equation*}%
Since the function $x\mapsto 1/\sqrt{x}$ is CM on $\left( 0,\infty \right) $%
, while $P_{1/2}\left( x\right) $ is CM on $\left( 0,\infty \right) $ in
view of Corollary \ref{C-B2}, it then follows from \cite[Theorem 1]%
{Miller-ITSF-12(4)-2001} that so is $x^{-1/2}P_{1/2}\left( x\right) $ on $%
\left( 0,\infty \right) $.
\end{remark}

We note that (\ref{T-xdK/K}) can be written as%
\begin{equation*}
1+\frac{1}{x}-\frac{K_{v-1}\left( x\right) K_{v+1}\left( x\right) }{%
K_{v}\left( x\right) ^{2}}=\frac{1}{x}\left( x+\frac{xK_{v}^{\prime }\left(
x\right) }{K_{v}\left( x\right) }\right) ^{\prime }=\frac{1}{x}\Lambda
^{\prime }\left( x\right) ,
\end{equation*}%
which, by $\Lambda ^{\prime }\left( x\right) >\left( <0\right) $ if $%
\left\vert v\right\vert >\left( <\right) 1/2$ given in Proposition \ref{P-B1}%
, we derive the following corollary.

\begin{corollary}
Let $v\in \mathbb{R}$ with $\left\vert v\right\vert \neq 1/2$. Then the
following inequality%
\begin{equation*}
\frac{K_{v-1}\left( x\right) K_{v+1}\left( x\right) }{K_{v}\left( x\right)
^{2}}<\left( >\right) 1+\frac{1}{x},
\end{equation*}%
or equivalently,%
\begin{equation}
K_{v}\left( x\right) ^{2}-K_{v-1}\left( x\right) K_{v+1}\left( x\right)
>\left( <\right) -\frac{1}{x}K_{v}\left( x\right) ^{2}  \label{T-I}
\end{equation}%
holds for $x>0$ if $\left\vert v\right\vert >\left( <\right) 1/2$.
\end{corollary}

\begin{remark}
The inequality (\ref{T-I}) is the Tur\'{a}n type inequality for modified
Bessel functions of the second kind, which first appeared in \cite%
{Baricz-EM-33-2015}. More such inequalities can be found in \cite%
{Ismail-SIAM-JMA-9(4)-1978}, \cite{Haeringen-JMP-19-1978}, \cite%
{Laforgia-JIA-2010-253035}, \cite{Baricz-BAMS-82-2010}, \cite%
{Segura-JMAA-374-2011}, \cite{Baricz-PAMS-141(2)-2013}, \cite%
{Baricz-FM-26(1)-2014} \cite{Baricz-EM-33-2015}, \cite{Yang-PAMS-145-2017}.
\end{remark}

Finally, we give an improvement of the double inequality (\ref{xdK/K<>}).

\begin{corollary}
Let $v\in \mathbb{R}$ with $\left\vert v\right\vert \neq 1/2$. Then the
following inequality%
\begin{equation}
-\sqrt{x^{2}+x+\max \left( \left\vert v\right\vert ,\frac{1}{2}\right) ^{2}}<%
\frac{xK_{v}^{\prime }\left( x\right) }{K_{v}\left( x\right) }<-\sqrt{%
x^{2}+x+\min \left( \left\vert v\right\vert ,\frac{1}{2}\right) ^{2}}
\label{xdK/K<>-a}
\end{equation}%
holds for $x>0$.
\end{corollary}

\begin{proof}
The desired inequalities are equivalent to%
\begin{eqnarray*}
-\sqrt{x^{2}+x+v^{2}} &<&\frac{xK_{v}^{\prime }\left( x\right) }{K_{v}\left(
x\right) }<-x-\frac{1}{2}\text{ if }\left\vert v\right\vert >\frac{1}{2}, \\
-x-\frac{1}{2} &<&\frac{xK_{v}^{\prime }\left( x\right) }{K_{v}\left(
x\right) }<-\sqrt{x^{2}+x+v^{2}}\text{ if }\left\vert v\right\vert <\frac{1}{%
2}.
\end{eqnarray*}%
By the double inequality (\ref{xdK/K<>}), it suffices to prove%
\begin{equation}
\frac{xK_{v}^{\prime }\left( x\right) }{K_{v}\left( x\right) }>\left(
<\right) -\sqrt{x^{2}+x+v^{2}}\text{ if }\left\vert v\right\vert >\left(
<\right) 1/2.  \label{xdK/K><}
\end{equation}%
To this end, we use the recurrence relations (see \cite[p. 79]%
{Watson-ATTBF-CUP-1922})%
\begin{eqnarray}
\frac{xK_{v}^{\prime }\left( x\right) }{K_{v}\left( x\right) }+v &=&-\frac{%
xK_{v-1}\left( x\right) }{K_{v}\left( x\right) }, \\
\frac{xK_{v}^{\prime }\left( x\right) }{K_{v}\left( x\right) }-v &=&-\frac{%
xK_{v+1}\left( x\right) }{K_{v}\left( x\right) }
\end{eqnarray}%
to get the identity%
\begin{equation*}
\left( \frac{xK_{v}^{\prime }\left( x\right) }{K_{v}\left( x\right) }\right)
^{2}-v^{2}=x^{2}\frac{K_{v-1}\left( x\right) K_{v+1}\left( x\right) }{%
K_{v}\left( x\right) ^{2}}.
\end{equation*}%
This in combination with the identity (\ref{T-xdK/K}) yields%
\begin{equation*}
x^{2}+x+v^{2}-\left( \frac{xK_{v}^{\prime }\left( x\right) }{K_{v}\left(
x\right) }\right) ^{2}=x\left( x+\frac{xK_{v}^{\prime }\left( x\right) }{%
K_{v}\left( x\right) }\right) ^{\prime }=x\Lambda ^{\prime }\left( x\right) .
\end{equation*}%
Using the inequality $\Lambda ^{\prime }\left( x\right) >\left( <\right) 0$
if $\left\vert v\right\vert >\left( <\right) 1/2$ due to Proposition \ref%
{P-B1} and noting $K_{v}\left( x\right) >0$ and $K_{v}^{\prime }\left(
x\right) <0$, the inequality (\ref{xdK/K><}) follows. This completes the
proof.
\end{proof}

\begin{remark}
The bound $-\sqrt{x^{2}+x+v^{2}}$ for $xK_{v}^{\prime }\left( x\right)
/K_{v}\left( x\right) $ given in inequality \ref{xdK/K><} is better than
some known ones, which refer to \cite[p. 242--243]{Baricz-EM-33-2015}. While
the another one $-x-1/2$ seems to be a new and simple one.
\end{remark}

\section{Conclusions}

In this paper, by the monotonicity rule for the ratio of two polynomials
with the same highest degree (Lemma \ref{L-PA/PB-pm}) and definition of
integral, we find that the decreasing (increasing) property of three ratios%
\begin{equation*}
\frac{F\left( x\right) }{G\left( x\right) }:=\frac{\int_{a}^{b}f\left(
t\right) e^{-xt}dt}{\int_{a}^{b}g\left( t\right) e^{-xt}dt}\text{, \ \ \ }%
\frac{\int_{0}^{\infty }f\left( t\right) e^{-xt}dt}{\int_{0}^{\infty
}g\left( t\right) e^{-xt}dt}\text{, \ \ \ }\frac{\int_{0}^{\infty }f\left(
t\right) e^{-x\mu \left( t\right) }dt}{\int_{0}^{\infty }f\left( t\right)
e^{-x\mu \left( t\right) }dt},
\end{equation*}%
under the condition that $f/g$ has the monotonicity pattern that $\nearrow
\searrow $ ($\searrow \nearrow $), according as $H_{F,G}\left( 0^{+}\right)
\geq \left( \leq \right) 0$. Otherwise, the three ratios $F/G$ are unimodal.
Since many special functions have integral representations in the form of $%
\int_{a}^{b}\left( \cdot \right) e^{-xt}dt$, our three theorems in this
paper are efficient tools of researching certain special functions.


\begin{thebibliography}{99}
\bibitem{Erdelyi-HTF-I-1981} A. Erd\'{e}lyi, W. Magnus, F. Oberhettinger,
F.G. Tricomi, Higher \emph{Transcendental Functions, vol. 1}, Krieger, New
York, 1981.

\bibitem{Watson-ATTBF-CUP-1922} G.N. Watson, \emph{A Treatise on the Theory
of Bessel Functions}, Cambridge University Press, Cambridge, 1922.

\bibitem{Simon-PDGRV-S-2006} M. K. Simon, \emph{Probability Distributions
Involving Gaussian Random Variables: A Handbook for Engineers and Scientists}%
, Springer, New York, 2006.

\bibitem{Miller-ITSF-12-2001} K. S. Miller and S. G. Samko, Completely
monotonic functions, \emph{Integr. Transf. Spec. Funct. }\textbf{12}\emph{\ }%
(2001), no. 4, 389--402.

\bibitem{Magnus-FTSFMPS-SV-1966} W. Magnus, F. Oberhettinger, and R.P. Soni, 
\emph{Formulas and Theorems for the Special Functions of Mathematical Physics%
}, Springer-Verlag, 1966.

\bibitem{Bernstein-AM-52-1929} S. Bernstein, Sur les fonctions absolument
monotones, \emph{Acta Math.} \textbf{52} (1929), 1--66.

\bibitem{Widder-TAMS-33-1931} D. V. Widder, Necessary and sufficient
conditions for the representation of a function as a Laplace integral, \emph{%
Trans. Amer. Math. Soc.} \textbf{33} (1931), 851--892.

\bibitem{Schilling-BFTA-2010} R. Schilling, R. Song, Z. Vondra\v{c}ek,
Bernstein functions: Theory and Applications, Studies in Mathematics, 
\textbf{37} (2010), de Gruyter, Berlin.

\bibitem{Yang-JIA-317-2017} Zh.-H. Yang and J. Tian, Monotonicity and
inequalities for the gamma function, \emph{J. Inequal. Appl.} \textbf{2017}
(2017): 317, https://doi.org/10.1186/s13660-017-1591-9.

\bibitem{Yang-JMAA-455-2017} Zh.-H. Yang, Some properties of the divided
difference of psi and polygamma functions, \emph{J. Math. Anal. Appl.} 
\textbf{455}, 761--777 (2017), http://dx.doi.org/10.1016/j.jmaa.2017.05.081.

\bibitem{Yang-arxiv-1409.6408} Zh.-H. Yang, A new way to prove L'Hospital
Monotone Rules with applications, arXiv:1409.6408 [math.CA].

\bibitem{Yang-JIA-2016-221} Zh.-H. Yang, W. Zhang and Yu.-M. Chu,
Monotonicity and inequalities involving the incomplete gamma function, \emph{%
J. Inequal. Appl.} \textbf{2016} (2016): 221, DOI 10.1186/s13660-016-1160-7.

\bibitem{Yang-JIA-2016-251} Zh.-H. Yang, W. Zhang and Yu.-M. Chu,
Monotonicity of the incomplete gamma function with applications, \emph{J.
Inequal. Appl.} \textbf{2016} (2016): 251, DOI 10.1186/s13660-016-1197-7.

\bibitem{Yang-JIA-2016-311} Zh.-H. Yang and Yu.-M. Chu, On approximating the
error function, \emph{J. Inequal. Appl.} \textbf{2016} (2016): 311, DOI
10.1186/s13660-016-1261-3.

\bibitem{Lv-JIA-2017-94} H.-L. Lv, Zh.-H. Yang, T.-Q. Luo and Sh.-Zh. Zheng,
Sharp inequalities for tangent function with applications, \emph{J. Inequal.
Appl.} \textbf{2017} (2017): 94, DOI 10.1186/s13660-017-1372-5.

\bibitem{Yang-MIA-20(3)-2017} Zh.-H. Yang and Y.-M. Chu, A monotonicity
property involving the ganeralized elliptic integral of the first kind, 
\emph{Math. Inequal. Appl.} \textbf{20} (2017), no. 3, 729--735
doi:10.7153/mia-20-46.

\bibitem{Yang-MIA-20(4)-2017} Zh.-H. Yang, W. Zhang and Yu.-M. Chu, Sharp
Gautschi inequality for parameter $0<p<1$ with applications, \emph{Math.
Inequal. Appl.} \textbf{20} (2017), no. 4, 1107--1120
doi:10.7153/mia-2017-20-71.

\bibitem{Luo-RM-72-2017} T.-Q. Luo, H.-L. Lv, Zh.-H. Yang, and Sh.-Zh.
Zheng, New sharp approximations involving incomplete gamma functions, \emph{%
Results Math.} \textbf{72} (2017), 1007--1020, DOI 10.1007/s00025-017-0713-5.

\bibitem{Yang-JMI-11-2017} Zh.-H. Yang and J. Tian, Optimal inequalities
involving power-exponential mean, arithmetic mean and geometric mean, \emph{%
J. Math. Inequal.} \textbf{11} (2017), no. 4, 1169--1183,
doi:10.7153/jmi-11-87.

\bibitem{Yang-JMI-12-2018} Zh.-H. Yang and J. Tian, Monotonicity and sharp
inequalities related to gamma function, \emph{J. Math. Inequal.} \textbf{12}
(2018), no. 1, 1--22, doi:10.7153/jmi-2018-12-01.

\bibitem{Yang-JMAA-428-2015} Zh.-H. Yang, Y.-M. Chu, M.-K. Wang,
Monotonicity criterion for the quotient of power series with applications, 
\emph{J. Math. Anal. Appl.} \textbf{428} (2015), 587--604.

\bibitem{Belzunce-I:M-E-40-2007} F. Belzunce, E. Ortega, and J. M. Ruiz, On
non-monotonic ageing properties from the Laplace transform, with actuarial
applications, \emph{Insurance: Mathematics and Economics} \textbf{40}
(2007), 1--14.

\bibitem{Xia-PJAM-7(2)-2016} F.-L. Xia, Zh.-H. Yang, and Y.-M. Chu, A new
proof for the monotonicity criterion of the quotient of two power series on
the infinite interval, \emph{Pacific Journal of Applied Mathematics} \textbf{%
7} (2016), no. 2, 97--101.

\bibitem{Yang-JIA-2015-157} Zh.-H. Yang, Y.-M. Chu and X.-H. Zhang,
Necessary and sufficient conditions for functions involving the psi function
to be completely monotonic, \emph{J. Inequal. Appl.} \textbf{2015} (2015):
157, DOI 10.1186/s13660-015-0674-8.

\bibitem{Yang-JIA-2015-299} Zh.-H. Yang and Y.-M. Chu, Inequalities for
certain means in two arguments, \emph{J. Inequal. Appl.} \textbf{2015}
(2015): 299, doi:10.1186/s13660-015-0828-8.

\bibitem{Yang-AAA-2014-702718} Zh.-H. Yang, Y.-M. Chu, and X.-J. Tao, A
double inequality for the trigamma function and its applications, \emph{%
Abstr. Appl. Anal.} \textbf{2014} (2014), Art. ID 702718, 9 pages.

\bibitem{Yang-arxiv-1705-05704} Zh.-H. Yang and J. Tian, Convexity and
monotonicity for the elliptic integrals of the first kind and applications,
arXiv:1705.05703 [math.CA].

\bibitem{Tims-MG-55-1971} S. R. Tims and J. A. Tyrrell, Approximate
evaluation of Euler's constant, \emph{Math. Gaz.} \textbf{55} (1971), 65--67.

\bibitem{Young-MG-75-1991} R. M. Young, Euler's constant, \emph{Math. Gaz.} 
\textbf{75} (1991), 187--190.

\bibitem{DeTemple-AMM-100-1993} D. W. DeTemple, A quicker convergence to
Euler's constant, \emph{Amer. Math. Monthly} \textbf{100} (1993), 468--470.

\bibitem{Negoi-GM-15-1997} T. Negoi, A faster convergence to the constant of
Euler, \emph{Gazeta Matematic\v{a}, seriac A} \textbf{15} (1997), 111--113
(in Romanian).

\bibitem{Alzer-AMSUH-68-1998} H. Alzer, Inequalities for the gamma and
polygamma functions, \emph{Abh. Math. Sem. Univ. Hamburg} \textbf{68}
(1998), 363-372.

\bibitem{Villarino-arXiv-0510585} M. B. Villarino, Sharp bounds for the
harmonic numbers, arXiv:math/0510585 [math.CA].

\bibitem{Villarino-JIPAM-9(3)-2008} M. B. Villarino, Ramanujan's harmonic
number expansion into negative powers of triangular number, \emph{J.
Inequal. Pure Appl. Math.} \textbf{9} (2008), no. 3, Art. 89, 12 pp.
Available online at
http://www.emis.de/journals/JIPAM/images/245\_07\_JIPAM/245\_07.pdf.

\bibitem{Mortici-CMA-59-2010} C. Mortici, On new sequences converging
towards the Euler-Mascheroni constant, \emph{Comput. Math. Appl.} \textbf{59}
(2010), 2610--2614.

\bibitem{Chen-AML-23-2010} C.-P. Chen, Inequalities for the Euler-Mascheroni
constant, \emph{Appl. Math. Lett.} \textbf{23} (2010), 161--164.

\bibitem{Qi-AMC-218-2011} F. Qi and B.-N. Guo, Sharp Bounds for Harmonic
Numbers, \emph{Appl. Math. Comput.} \textbf{218} (2011), 991--995,
doi:10.1016/j.amc.2011.01.089.

\bibitem{Chen-CMA-64-2013} C.P. Chen and C. Mortici, New sequence converging
towards the Euler--Mascheroni constant, \emph{Comput. Math. Appl.} \textbf{64%
} (2012), 391--398.

\bibitem{Lu-JMAA-419-2014} D. Lu, Some new convergent sequences and
inequalities of Euler's constant, \emph{J. Math. Anal. Appl.} \textbf{419}
(2014), 541--552.

\bibitem{Yang-AMC-268-2015} Zh.-H. Yang, Y.-M. Chu and X.-H. Zhang, Sharp
bounds for psi function, \emph{Appl. Math. Comput.} \textbf{268} (2015),
1055--1063, http://dx.doi.org/10.1016/j.amc.2015.07.012.

\bibitem{Zhao-JIA-193-2015} T.-H. Zhao, Zh.-H. Yang and Y.-M. Chu,
Monotonicity properties of a function involving the psi function with
applications, \emph{J. Inequal. Appl}. \textbf{2015} (2015): 193, DOI
10.1186/s13660-015-0724-2.

\bibitem{Baricz-BAMS-82-2010} \'{A}. Baricz, Tur\'{a}n type inequalities for
modified Bessel functions, \emph{Bull. Aust. Math. Soc.} \textbf{82} (2010),
no. 2, 254--264.

\bibitem{Yang-PAMS-145-2017} Zh.-H. Yang and Sh.-Zh. Zheng, The monotonicity
and convexity for the ratios of modified Bessel functions of the second kind
and applications, \emph{Proc. Amer. Math. Soc.} \textbf{145} (2017),
2943--2958.

\bibitem{Abramowitz-HMFFGMT-1972} M. Abramowitz and I. A. Stegun (Eds), 
\emph{Handbook of Mathematical Functions with Formulas, Graphs, and
Mathematical Tables}, National Bureau of Standards, Applied Mathematics
Series 55, Dover Publications, New York and Washington, 1972.

\bibitem{Miller-ITSF-12(4)-2001} K. S. Miller and S. G. Samko, Completely
monotonic functions, \emph{Integ. Transf. Special Funct.} \textbf{12}
(2001), no. 4, 389--402.

\bibitem{Bordelon-SIAMREV-15-1973} D. J. Bordelon, Problem 72-15,
inequalities for special functions, \emph{SIAM Rev.} \textbf{15} (1973),
665--668.

\bibitem{Ross-SIAMREV-15-1973} D. K. Ross, Problem 72-15, inequalities for
special functions, \emph{SIAM Rev. }\textbf{15} (1973), 668--670.

\bibitem{Paris-SIAM-JMA-15(1)-1984} R. B. Paris, An inequality for the
Bessel function $J_{v}\left( vx\right) $, \emph{SIAM J. Math. Analysis} 
\textbf{15} (1984), no. 1, 203--205.

\bibitem{Laforgia-JVAM-34(3)-1991} A. Laforgia, Bounds for modified Bessel
functions, \emph{J. Comput. Appl. Math.} \textbf{34} (1991), no. 3, 263--267.

\bibitem{Joshi-JAMS-50-1991} C. M. Joshi and S. K. Bissu, Some inequalities
of Bessel and modified Bessel functions, \emph{J. Austral. Math. Soc. A} 
\textbf{50} (1991), no. 2, 333--342.

\bibitem{Baricz-PEMS-53-2010} \'{A}. Baricz, Bounds for modified Bessel
functions of the first and second kinds, \emph{Proc. of the Edinburgh Math.
Soc.} \textbf{53} (2010), 575--599, DOI:10.1017/S0013091508001016.

\bibitem{Baricz-EM-33-2015} \'{A}. Baricz, Bounds for Tur\'{a}nians of
modified Bessel functions, \emph{Expo. Math.} \textbf{2015} (2015), no. 2,
223--251.

\bibitem{Ismail-SIAM-JMA-9(4)-1978} M.E.H. Ismail, M.E. Muldoon,
Monotonicity of the zeros of a cross-product of Bessel functions, \emph{SIAM
J. Math. Anal.} \textbf{9} (1978), no. 4, 759--767.

\bibitem{Haeringen-JMP-19-1978} H. van Haeringen, Bound states for r-2-like
potentials in one and three dimensions, \emph{J. Math. Phys.} \textbf{19}
(1978), 2171--2179.

\bibitem{Laforgia-JIA-2010-253035} A. Laforgia, P. Natalini, Some
inequalities for modified Bessel functions, \emph{J. Inequal. Appl.} \textbf{%
2010} (2010), Art. ID 253035, 10 pp.

\bibitem{Segura-JMAA-374-2011} J. Segura, Bounds for ratios of modified
Bessel functions and associated Tur\'{a}n-type inequalities, \emph{J. Math.
Anal. Appl.} \textbf{374} (2011), 516--528.

\bibitem{Baricz-PAMS-141(2)-2013} \'{A}. Baricz and S. Ponnusamy, On Tur\'{a}%
n type inequalities for modified Bessel functions, \emph{Proc. Amer. Math.
Soc.} \textbf{141} (2013), no. 2, 523--532.

\bibitem{Baricz-FM-26(1)-2014} \'{A}. Baricz and T. K. Pog\'{a}ny, Tur\'{a}n
determinants of Bessel functions, \emph{Forum Math.} \textbf{26} (2014), no.
1, 295--322.
\end{thebibliography}
\end{document}